\documentclass[11pt]{amsart}
\usepackage[letterpaper,margin=.6in]{geometry}
\usepackage[colorlinks,linkcolor=blue,citecolor=black!75!red]{hyperref}
\usepackage[matrix,arrow,ps,color,line,curve,frame,all]{xy}

\usepackage{tikz, tikz-cd}
\usetikzlibrary{matrix,arrows,decorations.pathmorphing, fit}
\tikzset{myboxgroup/.style={draw, densely dotted}} % style for the boxed groups

\usepackage{multirow}

\makeatletter
\renewcommand*\env@matrix[1][*\c@MaxMatrixCols c]{%
  \hskip -\arraycolsep
  \let\@ifnextchar\new@ifnextchar
  \array{#1}}
\makeatother

\usepackage{amsmath,amsthm,stmaryrd}
\usepackage{cleveref}

\newtheorem{lemma}{Lemma}[section]
\newtheorem{theorem}[lemma]{Theorem}
\newtheorem*{theorem*}{Theorem}
\newtheorem{proposition}[lemma]{Proposition}
\newtheorem{corollary}[lemma]{Corollary}

\theoremstyle{definition}

\newtheorem{definitionnodiamond}[lemma]{Definition}
\newtheorem{examplenodiamond}[lemma]{Example}
\newtheorem{remarknodiamond}[lemma]{Remark}
\newtheorem{conventionnodiamond}[lemma]{Convention}
\newtheorem{notationnodiamond}[lemma]{Notation}

\newenvironment{definition}{\begin{definitionnodiamond}}{\hfill\ensuremath\blacklozenge\end{definitionnodiamond}}

\newenvironment{remark}{\begin{remarknodiamond}}{\hfill\ensuremath\blacklozenge\end{remarknodiamond}}

\makeatletter
\let\xx@thm\@thm
\AtBeginDocument{\let\@thm\xx@thm}
\makeatother

\numberwithin{equation}{section}

\newcounter{stepofproof}

\crefname{section}{Section}{Sections}
\crefformat{section}{#2section~#1#3}
\Crefformat{section}{#2Section~#1#3}

\crefname{subsection}{}{Subsections}
\crefformat{subsection}{\S#2#1#3}
\Crefformat{subsection}{\S#2#1#3}

\crefname{definition}{Definition}{Definitions}
\crefformat{definition}{#2Definition~#1#3}
\Crefformat{definition}{#2Definition~#1#3}

\crefname{definitionnodiamond}{Definition}{Definitions}
\crefformat{definitionnodiamond}{#2Definition~#1#3}
\Crefformat{definitionnodiamond}{#2Definition~#1#3}

\crefname{example}{Example}{Examples}
\crefformat{example}{#2Example~#1#3}
\Crefformat{example}{#2Example~#1#3}

\crefname{examplenodiamond}{Example}{Examples}
\crefformat{examplenodiamond}{#2Example~#1#3}
\Crefformat{examplenodiamond}{#2Example~#1#3}

\crefname{remark}{Remark}{Remarks}
\crefformat{remark}{#2Remark~#1#3}
\Crefformat{remark}{#2Remark~#1#3}

\crefname{remarknodiamond}{Remark}{Remarks}
\crefformat{remarknodiamond}{#2Remark~#1#3}
\Crefformat{remarknodiamond}{#2Remark~#1#3}

\crefname{convention}{Convention}{Conventions}
\crefformat{convention}{#2Convention~#1#3}
\Crefformat{convention}{#2Convention~#1#3}

\crefname{notation}{Notation}{Notations}
\crefformat{notation}{#2Notation~#1#3}
\Crefformat{notation}{#2Notation~#1#3}

\crefname{notationnodiamond}{Notation}{Notations}
\crefformat{notationnodiamond}{#2Notation~#1#3}
\Crefformat{notationnodiamond}{#2Notation~#1#3}

\crefname{lemma}{Lemma}{Lemmas}
\crefformat{lemma}{#2Lemma~#1#3}
\Crefformat{lemma}{#2Lemma~#1#3}

\crefname{proposition}{Proposition}{Propositions}
\crefformat{proposition}{#2Proposition~#1#3}
\Crefformat{proposition}{#2Proposition~#1#3}

\crefname{corollary}{Corollary}{Corollaries}
\crefformat{corollary}{#2Corollary~#1#3}
\Crefformat{corollary}{#2Corollary~#1#3}

\crefname{theorem}{Theorem}{Theorems}
\crefformat{theorem}{#2Theorem~#1#3}
\Crefformat{theorem}{#2Theorem~#1#3}

\crefname{assumption}{Assumption}{Assumptions}
\crefformat{assumption}{#2Assumption~#1#3}
\Crefformat{assumption}{#2Assumption~#1#3}

\crefname{enumi}{}{}
\crefformat{enumi}{(#2#1#3)}
\Crefformat{enumi}{(#2#1#3)}

\crefname{equation}{}{}
\crefformat{equation}{(#2#1#3)}
\Crefformat{equation}{(#2#1#3)}

\crefname{align}{}{}
\crefformat{align}{(#2#1#3)}
\Crefformat{align}{(#2#1#3)}

\crefname{proofstep}{Step}{Steps}
\crefformat{proofstep}{#2Step~#1#3}
\Crefformat{proofstep}{#2Step~#1#3}

\crefname{table}{Table}{Tables}
\crefformat{table}{#2Table~#1#3}
\Crefformat{table}{#2Table~#1#3}

\newcommand\arXiv[1]{\href{http://arxiv.org/abs/#1}{\nolinkurl{arXiv:#1}}}
\newcommand\MRnumber[1]{\href{http://www.ams.org/mathscinet-getitem?mr=#1}{\nolinkurl{MR#1}}}
\newcommand\DOI[1]{\href{http://dx.doi.org/#1}{\nolinkurl{DOI:#1}}}
\newcommand\MAILTO[1]{\href{mailto:#1}{\nolinkurl{#1}}}

\usepackage{amsfonts,amssymb}

\newcommand\fd{\mathfrak d}

\newcommand\fm{\mathfrak m}
\newcommand\fp{\mathfrak p}

\def\sM{{\sf M}}
\def\sfM{{\sf M}}

\def\sff{{\sf f}}
\def\sfg{{\sf g}}
\def\sfp{{\sf p}}
\def\sfu{{\sf u}}
\def\sfv{{\sf v}}
\def\sfw{{\sf w}}
\def\sfx{{\sf x}}
\def\sfy{{\sf y}}

\newcommand\bP{\mathbb P}

\newcommand\bZ{\mathbb Z}

\def\a{\alpha}
\def\b{\beta}
\def\c{\gamma}
\def\d{\delta}

\def\l{\lambda}

\def\s{\sigma}

\def\G{\Gamma}

\def\FF{{\mathbb F}}

\def\PP{{\mathbb P}}

\def\NN{{\mathbb N}}
\def\PP{{\mathbb P}}

\def\ZZ{{\mathbb Z}}

\newcommand\cA{\mathcal A}

\newcommand\cC{\mathcal C}
\newcommand\cD{\mathcal D}

\newcommand\cO{\mathcal O}

\newcommand\cR{\mathcal R}

\newcommand\cV{\mathcal V}

\def\Hom{\operatorname{Hom}}
\def\Ext{\operatorname{Ext}}
\def\RHom{\operatorname{RHom}}
\def\Aut{\operatorname{Aut}}

\DeclareMathOperator\Spec{Spec}

\newcommand\gldim{\mathrm{gldim}}
\newcommand\id{\mathrm{id}}

\def\GL{{\sf GL}}
\def\Gr{{\sf Gr}}
\def\QGr{\operatorname{\sf QGr}}
\def\Tors{\operatorname{\sf Tors}}
\def\pdim{{\operatorname{pdim}}}

\def\Proj{\operatorname{Proj}}

\def\Mod{{\sf Mod}}
\def\pd{{\operatorname{\partial}}}

\renewcommand\lim{\varprojlim}

\DeclareMathOperator{\diag}{diag}
\DeclareMathOperator{\hdet}{hdet}

\newcommand{\isoto}{\xrightarrow{\smash{\raisebox{-0.25em}{$\sim$}}}}

\def\hdot{{\:\raisebox{2pt}{\text{\circle*{1.5}}}}}

\numberwithin{equation}{section}

\title[New AS-regular algebras via normal extensions]
{New Artin-Schelter regular and Calabi-Yau algebras via normal extensions}

\author[Alex Chirvasitu]{Alex Chirvasitu}
\author[Ryo Kanda]{Ryo Kanda} 	
\author[S. Paul Smith]{S. Paul Smith}

\address[Alex Chirvasitu]{Department of Mathematics, University at
  Buffalo, Buffalo, NY 14260-2900, USA.}  \email{achirvas@buffalo.edu}

\address[Ryo Kanda]{Department of Mathematics, Graduate School of
  Science, Osaka University, Toyonaka, Osaka, 560-0043, Japan; and
  Department of Mathematics, Box 354350, University of Washington,
  Seattle, WA 98195, USA.}  \email{ryo.kanda.math@gmail.com}

\address[S. Paul Smith]{Department of Mathematics, Box 354350,
  University of Washington, Seattle, WA 98195, USA.}
\email{smith@math.washington.edu}


%17B37 

\begin{document}

\pagenumbering{arabic}

\begin{abstract}
We introduce a new method to construct 4-dimensional Artin-Schelter regular algebras as normal extensions of (not necessarily noetherian) 3-dimensional ones. 
The method produces large classes of new 4-dimensional Artin-Schelter regular algebras.
When applied to a 3-Calabi-Yau algebra our method produces a flat family of central extensions of it that are 4-Calabi-Yau, and all 4-Calabi-Yau central extensions
having the same generating set as the original 3-Calabi-Yau algebra arise in this way. 
%Some of the 2-generated 4-dimensional Artin-Schelter regular algebras discovered by Lu-Palmieri-Wu-Zhang \cite{LPWZ} 
%can be obtained by our method and our results provide a new proof of their regularity.  
Each normal extension has the same generators as the original 3-dimensional algebra, and its relations consist of all but one of the relations for the original algebra and an equal number of new relations determined by ``the missing one'' and a tuple of scalars satisfying some numerical conditions. 
We determine the Nakayama automorphisms of the 4-dimensional algebras obtained by our method and as a consequence show that their 
homological determinant is 1. This supports the conjecture in \cite{MS16} that the homological determinant of the Nakayama automorphism is 1 for all 
Artin-Schelter regular connected graded algebras. Reyes-Rogalski-Zhang proved this is true in the noetherian case  \cite[Cor.~5.4]{RRZ17}.
\end{abstract}

\maketitle

\setcounter{tocdepth}{1} % depth for TOC
\tableofcontents
%\setcounter{tocdepth}{2} % depth for PDF bookmarks

%%%%%%%%%%%%%%%%%%%%%%%%%%%%%%%%%%%%%%%%%%%%%%%%%%%%%%%%%%%%%%%%%%%%%%%%%%%%%%%%%
%%%%%%%%%%%%%%%%%%%%%%%%%%%%%%%%%%%%%%%%%%%%%%%%%%%%%%%%%%%%%%%%%%%%%%%%%%%%%%%%%

\section{Introduction}\label{se.intro}

\subsection{Notation}
Throughout this paper we fix a field $\Bbbk$. All vector spaces will be $\Bbbk$-vector
spaces, all categories will be $\Bbbk$-linear, and all functors will be $\Bbbk$-linear.
All algebras will be finitely presented $\NN$-graded $\Bbbk$-algebras.  Such an algebra, $A=A_0 \oplus A_1 \oplus \cdots$,
is {\sf connected} if $A_0=\Bbbk$. All graded algebras in this paper
are connected and generated by their degree-one components.  Thus, we
are concerned with algebras $TV/(R)$ where $TV$ denotes the tensor
algebra on a finite dimensional vector space $V$ placed in degree $1$ 
and $(R)$ denotes the ideal generated by a finite dimensional graded
subspace $R$ of $TV$ whose elements have degrees $\ge 2$.

The algebras we study are not required to be noetherian or to have
finite Gelfand-Kirillov dimension.

\subsection{}
The assumption that a non-commutative algebra has finite global
dimension does not imply that it satisfies good analogues of the
Cohen-Macaulay and Gorenstein properties. For example, all free
algebras on a finite positive number of generators have global
dimension 1; for all $d \ge 0$, there are finite dimensional algebras having global dimension $d$; 
for all $d \ge 1$, there are algebras of global dimension $d$ whose only ideals are $\{0\}$ and the rings themselves;
and so on.  However, for graded algebras there are two ways to
strengthen the finite global dimension hypothesis that have proven
sufficiently effective that they lead to a good non-commutative
analogue of projective algebraic geometry.

The algebras satisfying these stronger definitions are called Artin-Schelter regular and twisted
Calabi-Yau algebras, respectively. The definitions vary a little from one
paper to another but in this paper they are defined in such a way
that they are, in fact, the same (see \cref{ssect.CY}).  A particularly important subset of these algebras are the Calabi-Yau algebras: 
their place among all twisted Calabi-Yau algebras is analogous to that occupied by the symmetric
algebras among all finite dimensional Frobenius algebras.

Many Artin-Schelter regular and Calabi-Yau algebras are known. The
3-dimensional Artin-Schelter regular algebras (meaning those of global dimension 3) that have finite
Gelfand-Kirillov dimension were classified by Artin-Schelter
\cite{AS87} and Artin-Tate-Van den Bergh \cite{ATV1}. Using those
classifications, the 3-dimensional Calabi-Yau algebras of finite
Gelfand-Kirillov dimension were classified by Mori-Smith \cite{ms-cy}
and Mori-Ueyama \cite{mu-cy}. A bewildering variety of 4-dimensional
Artin-Schelter regular and Calabi-Yau algebras are known but no
classification is in sight.
The construction of new examples is an important on-going research goal.

\subsection{The main result}

This paper introduces a new method for producing 4-dimensional
Artin-Schelter regular and Calabi-Yau algebras from 3-dimensional ones.

Before stating an approximate version of our main result we need the
following terminology.  An element $\Omega$ in a ring $D$ is {\sf
  normal} if $\Omega D=D\Omega$, and is {\sf regular} if it is not a zero-divisor.  Following the terminology in
\cite{LBSvdB}, we call  $D$ a {\sf normal extension by $\Omega$} 
of a ring $A$ if there is a regular normal element $\Omega$
in $D$ such that $A \cong D/(\Omega)$. If, in addition, $\Omega$ belongs to the center of $D$ we call $D$ a  {\sf central extension} of $A$.

Now let $A=TV/(R)$ be a $3$-dimensional Artin-Schelter regular algebra. Since $A$ is generated by its degree-one component we can, without loss of generality, 
assume that $R\subseteq V^{\otimes m}$ for some $m\geq 2$ (see \cref{ssect.mKoszul}). We do that.
Moreover,  Dubois-Violette proved  (see \cref{thm.DV}) that there is a twisted superpotential $\sfw\in V^{\otimes m+1}$ such that $A$ is the algebra $A(\sfw):=TV/(\pd_{1}\sfw,\ldots,\pd_{n}\sfw)$ where the $\pd_{i}\sfw$ are the partial derivatives of $\sfw$ with respect to a basis $x_{1},\ldots,x_{n}$ for $V$. Such $\sfw$ is uniquely determined up to scalar multiplication and the definition of $A(\sfw)$ does not depend on the choice of basis for $V$. See \S\S \ref{subsec.superpotential} and \ref{subsec.not.cycder} for details on twisted superpotentials.

With this notation, our main result is as follows.

\begin{theorem}
[Approximate version of \cref{th.3-stg}]
\label{thm.main}
Let $A=TV/(R)$ be a $3$-dimensional Artin-Schelter regular algebra and fix a basis $x_{1},\ldots,x_{n}$ for $V$. Let $\sfw\in V^{\otimes(m+1)}$ be a twisted superpotential such that $A=TV/(\pd_{1}\sfw,\ldots,\pd_{n}\sfw)$.
Fix an index $1\le k \le n$, let
$\sfp=(p_1,\ldots,p_n)$ be a ``good'' tuple of non-zero scalars, let
$D(\sfw,\sfp)$ be the free algebra $TV$ modulo the $2(n-1)$
relations
$$
\pd_i\sfw \;=\; x_i (\pd_k\sfw) - p_i (\pd_k \sfw) x_i \;=\; 0 \qquad (i=1,\ldots,n, \ i \ne k), 
$$
and let $\Omega$ denote the image of $\pd_k\sfw$ in $D(\sfw,\sfp)$. 
\begin{enumerate}
\item
$D(\sfw,\sfp)$ is a 4-dimensional Artin-Schelter regular algebra and is a normal extension of 
$A=D(\sfw,\sfp)/(\Omega)$.
   \item 
  $D(\sfw,\sfp)$ is noetherian if and only if $A$ is.
  \item 
  $D(\sfw,\sfp)$ has finite Gelfand-Kirillov dimension if and only if $A$ does.
    \item 
  $D(\sfw,\sfp)$ is a domain if $A$ is.
\end{enumerate}
\end{theorem}

In \Cref{thm.main}, the algebra $A$ has $n$ generators and $n$ defining relations, all of degree $m$, and the algebra $D(\sfw,\sfp)$
has $n$ generators and $2n-2$ relations, $n-1$ of degree $m$ and $n-1$ of degree $m+1$.

It is well-known that items (2)-(4) follow from (1) so the real content of \Cref{thm.main} is (1).

The condition that the tuple $\sfp$ is ``good'' can be stated in several ways, perhaps the simplest of which is 
\Cref{th.3-stg}\cref{item.eigenvec} (see \cref{re.gen}).

It is not immediately obvious that there are {\it any} ``good''  $\sfp$. However, as we show in \S\ref{sect.examples}, there {\it are} good $\sfp$'s for 
almost all the generic 3-dimensional Artin-Schelter regular algebras found by Artin and Schelter. Furthermore, if $A$ is a 3-Calabi-Yau algebra, then 
$\sfp=(1,\ldots,1)$ is good and in that case $D(\sfw,\sfp)$ is 4-dimensional Calabi-Yau algebra and is a central extension of $A$.

\begin{corollary}
[\cref{cor.cy}]
\label{cor.main}
  Let $TV/(R)=A(\sfw)$ be a 3-Calabi-Yau algebra and fix an index $1\leq k\leq n$. Then  $TV$ modulo the relations
  \begin{equation*}
    \pd_i\sfw \;=\;  [x_i,\partial_k{\sfw}] \;=\; 0 \qquad (i=1,\ldots,n,\ i\neq k)
  \end{equation*}
  is a $4$-Calabi-Yau algebra and is a central   extension of $A(\sfw)$ by the image of $\pd_{k}\sfw$. 
  \end{corollary}
  
  Although the terminology ``central extension'' and ``normal extension'' follows that in \cite{LBSvdB} there is no overlap between this paper and that one.
  In this paper the natural map $D \to A$ from one of the normal or central extensions in \cref{thm.main} is an isomorphism in degree 1,
   i.e., $A$ and $D$ have the ``same''   generating sets, whereas the corresponding map $D \to A$ in \cite{LBSvdB} has a 1-dimensional kernel in degree 1.
   
\subsection{A flat family of 4-Calabi-Yau algebras}   
Although it is not  immediately apparent, if $A(\sfw)$ is fixed, then \cref{cor.main} can be used to produce 
a flat family of 4-Calabi-Yau algebras over $\PP^{n-1}$ that map onto $A(\sfw)$. 
To do this we must state  \cref{cor.main} in a basis-free way. 

Let $V$ be an $n$-dimensional $\Bbbk$-vector space  and let $m \in \ZZ_{\ge 1}$. For each $x \in V$, let $x^\perp:=\{\psi \in V^* \; | \; \psi(x)=0\}$.
Define   $\langle -,-\rangle : V^* \times V^{\otimes (m+1)} \to V^{\otimes m}$ by $\langle \psi,\sfw \rangle:=(\psi \otimes \id^{\otimes m})(\sfw)$.

\begin{theorem}
\label{thm2}
Let $A(\sfw)$ be a 3-Calabi-Yau algebra where $\sfw\in V^{\otimes (m+1)}$ is a superpotential. For each $x \in V$, let $I_x$ be the ideal in $TV$ generated by 
$\{\langle \psi,\sfw \rangle \; | \; \psi \in x^\perp\}$ and $ \{[y, \langle \psi,\sfw \rangle] \; | \; y \in V, \, \psi \in V^*\}$. Define $D_x=TV/I_x$. By definition $D_{0}=A(\sfw)$.
\begin{enumerate}
	\item\label{thm2.flatfamily} $\{D_x \; | \; x \in \PP^{n-1}=\PP(V)\}$ is a 
flat family of  $4$-Calabi-Yau algebras  having
degree-$m$ central regular elements $\Omega_x$ such that  $A(\sfw) = D_x/(\Omega_x)$ for all $x\in V$.
	\item\label{thm2.allcyext} All degree-$m$ central extensions of $A(\sfw)$ are obtained in this way and they are $4$-Calabi-Yau (\cref{th.allcyext}).
\end{enumerate}
\end{theorem}

In \S\ref{sect.examples}, we give some explicit examples to show that
not every normal extension of a 3-Calabi-Yau algebra need be a central
extension.

\subsection{}\label{subse.fl-nc}

Given part \cref{thm2.flatfamily} of \Cref{thm2} above, it is natural to consider the algebras $D_x$ as homogeneous coordinate rings
of a well-behaved family of non-commutative schemes. In order to make this precise, we need some
preparation.

Let $\cD$ be the sheaf of $\cO_{\PP(V)}$-algebras on $\PP(V)$ defined
in \Cref{subse.fam} below.  The fiber of $\cD$ over a $\Bbbk$-point
$x \in \PP(V)$ is the algebra $D_x$ defined in \Cref{thm2}.  For each
open subscheme $U\subseteq \PP(V)$ we denote by $\cD_U$ the
restriction of $\cD$ to $U$.

\begin{definition}\label{def.fam}
If $D$ is a connected graded $\Bbbk$-algebra we write $\Gr(D)$ for the category of $\ZZ$-graded left $D$-modules
and $\QGr(D)$ for the quotient category $\Gr(D)/\Tors(D)$ where $\Tors(D)$ is the full subcategory of 
graded $D$-modules that are the sum of their finite dimensional submodules.

Similarly,   $\Gr(\cD_{U})$ is the category of graded quasi-coherent sheaves of $\cD_U$-modules and $\QGr(\cD_{U})$ is the 
quotient  by the subcategory $\Tors(\cD_{U})$ consisting of  modules  whose fibers are in $\Tors(D_x)$ for all $x\in \PP(V)$.
\end{definition}

In \cite[Defn.~3.2]{lvdb-def}, Lowen and Van den Bergh introduced the notion of a flat abelian $R$-linear
category over a commutative ring $R$.

\begin{theorem}\label{th.fl-nc}
  For every open affine $\Spec(R)=U\subseteq  \PP(V)$, the $R$-linear abelian
  category $\QGr(\cD_U)$ is flat over $R$ in the sense of \cite{lvdb-def}.
  We can express this by saying that $\QGr(\cD)$ is flat over $\PP(V)$.
\end{theorem}

The intuition here is that the non-commutative scheme represented by the category $\QGr(\cD)$
aggregates its fibers, the non-commutative schemes represented by $\QGr(D_x)$, $x\in \PP(V)$, into a
flat family over  $\PP(V)$.

\subsection{}
\label{sect.1331}
The 3-dimensional Artin-Schelter regular algebras of finite
Gelfand-Kirillov dimension (or, equivalently, the noetherian ones)
have been classified. They fall into two broad classes: those on 3
generators subject to 3 quadratic relations and those on 2 generators
subject to 2 cubic relations.  Applying \cref{thm.main} to these
produces 4-dimensional Artin-Schelter regular algebras on 3 generators
subject to two quadratic and two cubic relations, and 4-dimensional
Artin-Schelter regular algebras on 2 generators subject to one cubic
and one quartic relation, respectively.

Following the terminology in \cite{LPWZ}, we sometimes say that these
two classes of 3-dimensional Artin-Schelter regular algebras have
types $(1331)$ and $(1221)$, respectively. The terminology comes from
the ranks of the free modules that appear in the minimal projective
resolution of the trivial module (see \cref{def.as}).
%\footnote{It is no accident that these
%  numbers form a palindromic sequence. The definition of a
%  $d$-dimensional Artin-Schelter regular algebra says, in part, that
%  it satisfies a Gorenstein-like property which ensures that
%  $\Ext_A^i(A,\Bbbk) \cong \Ext_A^{d-i}(\Bbbk,A)^*$ and this, in turn,
%  implies that these ranks form a palindromic sequence.}
Extending
this terminology to the 4-dimensional case in the obvious way,
\cref{thm.main} produces 4-dimensional Artin-Schelter regular algebras
of types $(13431)$ and $(12221)$, respectively.

Once one drops the requirement that the 3-dimensional Artin-Schelter
regular algebra have finite Gelfand-Kirillov dimension (or,
equivalently, be noetherian), other ``types'' can appear. For example,
in \cite{e}, Eisenschlos constructs 3-Calabi-Yau algebras of
type $(1,n,n,1)$ for many $n \ge 3$.  \cref{thm.main} applies to
these algebras and produces 4-Calabi-Yau algebras of type $(1,n,2n-2,n,1)$.

Normal extensions of Artin-Schelter regular algebras of types $(1331)$ and $(1221)$ were also studied by Cassidy \cite{Cas99}. He obtained an analogous result to \cref{th.3-stg} for those two classes of algebras without any assumption on the matrix $Q$. The proof in \cite{Cas99} was given separately for each type, while our proof is given for all types of 3-dimensional regular algebras, which include those do not have finite Gelfand-Kirillov dimensions, simultaneously.

Most of the algebras produced by \cref{thm.main} are new. A few are not. For example, two of the four families of 4-dimensional regular algebras introduced by Lu-Palmieri-Wu-Zhang \cite{LPWZ} can be obtained by \cref{thm.main}. Thus the regularity for their algebras follows from \cref{thm.main}; in contrast, the proof of their
regularity in \cite{LPWZ} proceeded by finding an explicit basis of the defining ideal using Bergman's diamond lemma \cite{bgm}, or by observing that some of those algebras are Ore extensions of skew polynomial rings. The algebras produced by 
 \cref{thm.main} need not have PBW bases.

\subsection{Artin-Schelter regular algebras}
\label{def.as}
When $A$ is connected graded, $\Bbbk=A/A_{\ge 1}$ is a graded left
$A$-module concentrated in degree 0. We call it the {\sf trivial
  module}.

Let $d$ be a non-negative integer.  A connected graded algebra $A$ is
{\sf Gorenstein of dimension $d$} if it has left injective dimension
$d$ as a module over itself and
\begin{equation}
\label{AS-Gor-cond}
\Ext^i_A(\Bbbk,A) \cong \begin{cases} 
					\Bbbk(\ell) & \text{if $i=d$}
					\\
					0 & \text{if $i \ne d$}
				\end{cases}
\end{equation}
for some $\ell \in \ZZ$. The number $\ell$ in \cref{AS-Gor-cond} is
called the {\sf Gorenstein parameter} of $A$.

  A connected graded algebra $A$ is {\sf Artin-Schelter regular} (or
  {\sf AS-regular}, or simply {\sf regular})  of dimension $d$ if it
  \begin{enumerate}
    \renewcommand{\labelenumi}{(\arabic{enumi})}
  \item has finite global dimension $d$ and 
  \item is Gorenstein of dimension $d$.
  \end{enumerate}
  
  The Gorenstein and AS-regularity properties are known to be left-right symmetric. The projective resolution of $\Bbbk$ as a graded right $A$-module can be obtained by applying $\Hom_{A}(-,A)(-\ell)$. This also explains why the types of regular algebras are palindromic sequences; see the proof of \cite[Thm.~1.5]{AS87}. The left and right global dimensions coincide since both are equal to the projective dimension of $A$ as a bimodule (\cite[Lem.~4.3]{YZ06}; see also \cref{ssect.CY} below).
 
  Commutative Artin-Schelter regular algebras are polynomial rings; see \cite[Prop.~2.2]{Mat75}.

 The notion of regularity we use does not require or imply that the algebra be noetherian or have finite Gelfand-Kirillov dimension.
 However, by \cite[Prop.~3.1(1)]{SZ97}  it does imply that ${}_A\Bbbk$ has a finite projective resolution in which every term is a finitely generated free left $A$-module.

\subsection{Twisted Calabi-Yau algebras}
\label{ssect.CY}
Let $A$ be a graded algebra. We denote its opposite algebra by $A^\circ$  and its enveloping algebra by $A^{e}=A\otimes_{\Bbbk}A^\circ$. An $A$-bimodule $M$ is regarded as a left $A^{e}$-module via $(a\otimes b)\cdot x=axb$ for all $a,b\in A$ and $x\in M$. If $\nu$ is an automorphism of $A$ we denote by ${{}_{{}_\nu}} A_{{}_1}$ the left $A^e$-module that is $A$ as a graded vector space with action $(a \otimes b)\cdot c=\nu(a)cb$. We say that $A$ is {\sf twisted Calabi-Yau} of {\sf dimension} $d$ (abbreviated as \textsf{$d$-tCY} or simply {\sf tCY}) if $A$, as left $A^{e}$-module, has a finite-length projective resolution in which each term is finitely generated, and there is an isomorphism
\begin{equation}\label{CY-cond}
	\Ext_{A^{e}}^{i}(A,A^{e})\cong
	\begin{cases}
		{}_{\nu}A_{1}(\ell) & \text{if $i=d$}\\
		0 & \text{if $i\neq d$}
	\end{cases}
\end{equation}
of graded right $A^{e}$-modules for some $\ell\in\ZZ$ and some graded algebra automorphism $\nu$ of $A$. Such a $\nu$ is uniquely determined and is called the {\sf Nakayama automorphism} of $A$. We also denote it by $\nu_{\!{}_A}$ to clarify the algebra. (Some authors call $\nu^{-1}$ the Nakayama automorphism; see, e.g., \cite{RRZ14}.) We say that $A$ is {\sf $d$-Calabi-Yau} (or simply {\sf Calabi-Yau}, abbreviated as {\sf $d$-CY} or {\sf CY}) if it is twisted Calabi-Yau of dimension $d$ and $\nu=\id_{A}$.

\begin{proposition}
\cite[Lem.~1.2]{RRZ14}
\label{lem.RRZ}
A connected graded algebra is an Artin-Schelter regular algebra of dimension $d$ if and only if it is a twisted Calabi-Yau algebra of dimension $d$. If this is the case, the numbers $\ell$ in \cref{AS-Gor-cond} and \cref{CY-cond} are the same.
\end{proposition}

\subsection{}
The next result, which summarizes some of the results in \S\ref{sect.norm.ext}, applies to a larger class of algebras than those covered by \Cref{thm.main}.

\begin{theorem}\label{thm.nak-ad}
  Let $D$ be a connected graded algebra generated as a $\Bbbk$-algebra by $D_1$.
  Let $\Omega\in D$ be a homogeneous normal regular element of degree $\geq 2$ and let $A = D/(\Omega)$.
 If $A$ is twisted Calabi-Yau then so is $D$ and in that case the following results hold:
\begin{enumerate}
  \item 
  on the common space of generators $A_1=D_1$,  $ \nu_{\!{}_A} = \tau^{-1} \nu_{\!{}_D}$ where $\tau$ is the unique automorphism of $D$ such that 
  $\Omega x=\tau(x)\Omega$ for all $x \in D$;
  \item 
  $\Omega$ is an eigenvector for the Nakayama automorphism $\nu_{\!{}_D}$;
  \item 
if $\sigma$ is a degree-preserving automorphism of $D$ such that $\s(\Omega)=\l \Omega$,
then $\hdet(\s)$,  the homological determinant of $\s$ acting on $D$, is  $\lambda\hdet(\sigma|_{A})$.
\end{enumerate}
\end{theorem}

\subsection{$m$-Koszul algebras}
\label{ssect.mKoszul}
Because they are generated by their  degree-one components the 3-dimensional Artin-Schelter regular algebras $A(\sfw)$ in this paper are $m$-Koszul by \cite[Thm.~1.5]{AS87}. A connected graded algebra $A$ is {\sf $m$-Koszul} if it is finitely
presented and its relations are homogeneous of degree $m$ and $\Ext^i_A(\Bbbk,\Bbbk)$ is concentrated in degree $n(i)$ for all $i$, where $n(2j)=jm$ and $n(2j+1)=jm+1$ for integers $j$.
The class of $m$-Koszul algebras was identified and first studied by R. Berger \cite{Berger01}, the justification being that they are a natural generalization of
Koszul algebras (the $2$-Koszul algebras are exactly the Koszul algebras in the ``classical'' sense) and that there are many ``nice'' and ``natural''
$m$-Koszul algebras that are not $2$-Koszul.

Reyes-Rogalski-Zhang \cite[Cor.~5.4]{RRZ17} proved that the homological determinant of the Nakayama automorphism of a
noetherian Artin-Schelter regular connected graded algebras is 1. Mori-Smith \cite{MS16}  proved a similar result in the non-noetherian case when the algebra is
$m$-Koszul (and Artin-Schelter regular and connected) and conjectured that the equality holds without the 
$m$-Koszul assumption. It follows from \Cref{thm.nak-ad}(3) that the conjecture is true for the algebras $D(\sfw,\sfp)$
(which are never $m$-Koszul and often not noetherian).

\subsection{The structure of the paper}
In \S\ref{sect.norm.ext}, we prove some new results about the homological properties of a connected graded algebra and a normal extension of it.
The main theorem is stated in \S\ref{sect.main.thm} and some of its consequences are stated and proved there. Its proof is postponed to \S\ref{sec.proof}.
In the last section, \S\ref{sect.examples}, we determine the good tuples $\sfp$ for all the generic 3-dimensional Artin-Schelter regular algebras
of finite Gelfand-Kirillov dimension  listed in \cite[Tables (3.9) and (3.11)]{AS87}. 
This provides a host of new noetherian 4-dimensional Artin-Schelter regular algebras on 2 and 3 generators.
Some new normal and central extensions of non-noetherian 3-Calabi-Yau algebras are also presented in  \S\ref{sect.examples}.

%%%%%%%%%%%%%%%%%%%%%%%%%%%%%%%%%%%%%%%%%%%%%%%%%%%%%%%%%%%%%%%%%%%%%%%%%%%%%%%%%
\subsection*{Acknowledgements}

We would like to thank Izuru Mori for pointing out errors in an earlier version, Michaela Vancliff for providing information on Cassidy's work \cite{Cas99}, and the anonymous referee for a very careful reading and numerous insightful comments.  

A.C. was partially supported by NSF grants DMS-1565226 and DMS-1801011. R.K. was a JSPS Overseas Research Fellow and supported by JSPS KAKENHI Grant Numbers JP17K14164 and JP16H06337.

%%%%%%%%%%%%%%%%%%%%%%%%%%%%%%%%%%%%%%%%%%%%%%%%%%%%%%%%%%%%%%%%%%%%%%%%%%%%%%%%%
%%%%%%%%%%%%%%%%%%%%%%%%%%%%%%%%%%%%%%%%%%%%%%%%%%%%%%%%%%%%%%%%%%%%%%%%%%%%%%%%%
\section{Results on normal extensions}
\label{sect.norm.ext}

\subsection{}
\label{sect.norm.reg.elt}
In this section we are concerned about the relation between a  ring $D$ having a normal regular element $\Omega$ and 
the quotient $A:=D/(\Omega)$. 

To do this we make use of the automorphism $\tau:D \to D$ defined by the requirement that
  \begin{equation}\label{eq:tau}
    \Omega x = \tau(x)\Omega 
  \end{equation}
  for all $x \in D$. Since $\tau$ fixes $\Omega$, it descends to an automorphism of $A$ that we also denote by $\tau$.
  
Let $\sigma$ be an automorphism of a ring $D$. If $M$ is a left $D$-module we write ${}_\sigma M$ for the left $D$-module that is the
 abelian group $M$ endowed with the action 
 $$
 x \cdot m \; :=\; \sigma(x) m
 $$
 for all $x \in D$ and $m \in M$. There is an auto-equivalence $M \rightsquigarrow {}_\s M$ of the category $\Mod(D)$ of left $D$-modules.  
 The inverse $\sigma^{-1}:D \to D$ is an isomorphism ${}_\sigma D \to D$ of left $D$-modules.
Likewise, if $\sigma$ descends to $A$, then $\sigma^{-1}:A \to A$ is an isomorphism ${}_\sigma A \to A$ of left $D$-modules.
 
 We extend the notation in the previous paragraph to $D$-bimodules: if $\a$ and $\b$ are automorphisms of $D$ and $M$ is a $D$-bimodule, we write ${}_\a M_\b$ for the 
 $D$-bimodule that is $M$ as an abelian group endowed with the action
 $$
 x \cdot m \cdot y \; :=\; \a(x) m\b(y).
 $$
If $\a^{-1}\b=\c^{-1}\d$, then $\c\a^{-1}: {}_\a D_\b \to {}_\c D_\d$ is an isomorphism of $D$-bimodules. For example, ${}_\a D_1\cong {}_1D_{\a^{-1}}$. 

\subsection{}
In the situation of \cref{sect.norm.reg.elt}, the projective resolution 
\begin{equation}
\label{eq:res-of-A.1}
  \begin{tikzcd}
 0 \arrow[r]  & D  \arrow[r,"\cdot \Omega"] & D\arrow[r] & A\arrow[r] & 0
  \end{tikzcd}
\end{equation}
of $A$ as a left $D$-module is, in fact, an exact sequence
\begin{equation}
\label{eq:res-of-A.2}
  \begin{tikzcd}
 0  \arrow[r]   & {}_1D_\tau  \arrow[r,"\cdot \Omega"] & D\arrow[r] & A\arrow[r] & 0
   \end{tikzcd}
\end{equation}
 in the category of $D$-bimodules. The unadorned $D$ and $A$ in this sequence are given their natural $D$-bimodule structures.
 
 If $M$ is a left $D$-module, the result of applying $\Hom_D(-,M)$ to \cref{eq:res-of-A.2} is the exact sequence 
\begin{equation}
\label{eq:res-of-A.3}
  \begin{tikzcd}
 0  \arrow[r]   & \Hom_D(A,M)   \arrow[r]   & \Hom_D({}_1D_1,M)  \arrow[r,"\Omega \cdot "] & \Hom_D({}_1D_\tau,M) \arrow[r] & \Ext^1_D(A,M) \arrow[r] & 0.
   \end{tikzcd}
\end{equation}
 The map $\Phi:{}_\tau M \to  \Hom_D({}_1D_\tau,M)$ given by $\Phi(m)(d)=dm$ is an isomorphism of left $D$-modules when 
 $ \Hom_D({}_1D_\tau,M)$ is given the left $D$-module structure coming from the right action of $D$ on ${}_1D_\tau$.
 Thus,  $\Ext^1_D(A,M)$ is isomorphic as a left $D$-module to ${}_\tau (M/\Omega M)$.

\begin{proposition}\label{pr.enh}
  Let $D$ be a connected graded algebra, $\Omega\in D_m$ a normal element of degree $m$, and let $A = D/(\Omega)$.
Consider the following statements.
  \begin{enumerate}
  \renewcommand{\labelenumi}{(\arabic{enumi})}
\item\label{item.enh.elm} 
$\Omega$ is regular.
\item\label{item.enh.h} 
The Hilbert series of $A$ and $D$ are related by
  $h_D(t)=h_A(t)\left(1-t^m\right)^{-1}$.
\item\label{item.enh.alg} If $A$ is AS-regular of dimension $d$, then $D$ is AS-regular of dimension $d+1$.
\end{enumerate}
There are implications $\cref{item.enh.elm}\Leftrightarrow\cref{item.enh.h}\Rightarrow\cref{item.enh.alg}$
\end{proposition}
\begin{proof}
\cref{item.enh.elm} $\Leftrightarrow$ \cref{item.enh.h} 
This follows from the fact that the complex $0 \longrightarrow D(-m) \stackrel{\cdot \Omega}{\longrightarrow} D \longrightarrow A \longrightarrow 0$
is exact if and only if $\Omega$ is regular.

\cref{item.enh.elm} $\Rightarrow$ \cref{item.enh.alg} As explained in
\cite[\S2]{ATV1}, the global dimensions of each of $A$ and $D$ equal
the projective dimension of the trivial module $\Bbbk$ over the
respective algebra. Since $\Omega$ is regular and $\pdim({}_A\Bbbk)<\infty$,
\cite[Thm.~7.3.5(i)]{mr} implies that $\pdim({}_D\Bbbk)$ equals
$\pdim({}_A\Bbbk)+1 = d+1$. Thus, $\gldim(D)=d+1$.
The Gorenstein property for $D$ is equivalent to that of $A$ via the
isomorphism
  \begin{equation*}
    \Ext^{p+1}_D(\Bbbk,D)\cong \Ext^p_A(\Bbbk,A)
  \end{equation*}
  which is a consequence of \Cref{le.hom} below applied to $N=\Bbbk$ and
  $M=D$. 
\end{proof}

\begin{lemma}\label{le.hom}
  Let $D$ be a connected graded algebra, $\Omega\in D_m$ a normal element of degree $m$, and let $A = D/(\Omega)$.
Then, for every left $A$-module $N$ and all left $D$-modules $M$ on which $\Omega$ acts regularly, there are isomorphisms
  \begin{equation*}
    \Ext^{p+1}_D(N,M)\cong \Ext^p_A(N,{}_\tau (M/\Omega M))
  \end{equation*}
  which is functorial on both $N$ and $M$, for all $p$.
\end{lemma}
\begin{proof}
  The functor $\Hom_D(N,-)$ from left $D$-modules to vector spaces
  factors as
  \begin{equation*}
  \begin{tikzpicture}[auto,baseline=(current  bounding  box.center)]
    \path[anchor=base] (0,0) node (l) {$\Mod(D)$} +(6,0) node (r) {${\sf Vect}$.} +(3,-.5) node (d) {$\Mod(A)$};
    \draw[->] (l) to[bend left=6] node[pos=.5,auto] {$\scriptstyle \Hom_D(N,-)$} (r);
    \draw[->] (l) to[bend right=6] node[pos=.5,auto,swap] {$\scriptstyle \Hom_D(A,-)$} (d);
    \draw[->] (d) to[bend right=6] node[pos=.5,auto,swap] {$\scriptstyle \Hom_A(N,-)$} (r);    
  \end{tikzpicture}
\end{equation*}
Since the lower left functor preserves injectivity, \cite[Thm.~III.7.1]{gm} gives an isomorphism
\begin{equation}\label{eq.comp.rhom}
  \mathrm{RHom}_D(N,M)\cong \mathrm{RHom}_A(N,\mathrm{RHom}_D(A,M))
\end{equation}
in the left-bounded derived category of vector spaces, which is functorial on both $M$ and $N$.

Since $\Omega$ acts faithfully on $M$, $\Hom_D(A,M)=0$. As observed above, the canonical map $\Hom_D({}_1 D_\tau,M) \to \Ext^1_D(A,M)$ induces an
isomorphism $ {}_\tau(M/\Omega M) \isoto \Ext^1_D(A,M)$ of left $D$-modules.
Since $\pdim({}_D A)=1$, $\Ext^p_D(A,M)=0$ for $p\ge 2$. Hence
\begin{equation*}
  \mathrm{RHom}_D(A,M)\cong {}_\tau(M/\Omega M)[-1]
\end{equation*}
in the left-bounded derived category of left $A$-modules. Thus, 
\begin{equation*}
  \mathrm{RHom}_D(N,M)\cong \mathrm{RHom}_A(N,{}_\tau(M/\Omega M))[-1],
\end{equation*}
which is nothing but a reformulation of the conclusion of the lemma.
\end{proof}

On occasion, one might want to impose growth conditions on connected graded algebras that complement the good homological
properties in \Cref{def.as} (the more prevalent definition of Artin-Schelter 
regularity, for instance, requires finite Gelfand-Kirillov dimension;
see, e.g., \cite[Introduction]{AS87} or \cite[(2.12)]{ATV1}). The
following remark supplements \Cref{pr.enh} in that direction.

\begin{proposition}\label{le.growth} 
  Let $D$ be a connected graded algebra and $\Omega\in D_m$ a normal element of degree $m\geq 1$, and let $A = D/(\Omega)$.
  Then
  \begin{enumerate}
  \item\label{item.growth.gk} $A$ has finite Gelfand-Kirillov dimension if and only if $D$ does;
  \item\label{item.growth.noeth} $A$ is left or right noetherian if and only if $D$ is.
  \end{enumerate}
\end{proposition}
\begin{proof}
\cref{item.growth.gk} The finite Gelfand-Kirillov dimension of $D$ certainly entails that of its quotient $A= D/(\Omega)$.  Conversely, it follows from the exact sequence
 $D(-m) \stackrel{\cdot \Omega}{\longrightarrow} D \longrightarrow A \longrightarrow 0$ 
that
  \begin{equation}\label{eq:diff}
    \mathrm{dim}(D_k)- \mathrm{dim}(D_{k-m})\le \mathrm{dim}(A_k) 
  \end{equation}
  for all $k$. Thus,  if $\mathrm{dim}(A_k)$ has polynomial growth so
  does $\mathrm{dim}(D_k)$.

\cref{item.growth.noeth}
Certainly, if $D$ is noetherian so is $A$. The converse follows from
  \cite[Lem.~8.2]{ATV1}.
\end{proof}

Later, we determine the relationship between the
Nakayama automorphism of $A$ and that of $D$ whenever the equivalent
conditions \cref{item.enh.elm} and \cref{item.enh.h} of \Cref{pr.enh} are satisfied. As a first step
in that direction, we record the following consequence of \Cref{le.hom}.

\begin{lemma}\label{le.aa-quot}
  Let $D$ be a connected graded algebra, $\Omega\in D_m$ a normal  regular element of degree $m$, and let $A = D/(\Omega)$.
There is an  isomorphism
  \begin{equation*}
    \Ext^{p+2}_{D^e}(A,D^e)  \;  \cong  \; {}_\tau\Ext^p_{A^e}(A,A^e){}_{\tau^{-1}}
  \end{equation*}
  of $D$-bimodules.
\end{lemma}
\begin{proof}
We write ${}_{\tau}A^e{}_{\tau^{-1}}$ for $A^e$ endowed with the $D$-bimodule structure given by
$$
c \cdot(a \otimes b) \cdot d \;=\; \tau(\overline{c})a \otimes b\tau^{-1}(\overline{d})
$$
where $\overline{c}$ and $\overline{d}$ are the images in $A$ of $c,d \in D$ and $b\tau^{-1}(\overline{d})$
 is the product $b \times \tau^{-1}(\overline{d})$ in $A$ viewed as an element of $A^\circ$. As a left $D^e$-module,   ${}_{\tau}A^e{}_{\tau^{-1}}$ 
 is annihilated by $1 \otimes \Omega$ and $\Omega \otimes 1$ so ${}_{\tau}A^e{}_{\tau^{-1}}$ is a left 
 $A^e$-module (and, equivalently, an $A$-bimodule) where the subscripts  $\tau$ and $\tau^{-1}$ should now be viewed as the automorphisms of 
 $A$ induced by the automorphisms $\tau$ and $\tau^{-1}$ of $D$.

  Apply \Cref{le.hom} twice, first to the quotient  $D\otimes A^{\circ}=D^e/(1 \otimes \Omega)$, then to 
  $A^e=(D\otimes A^{\circ})/(\Omega \otimes 1)$. This results in an
  isomorphism
  \begin{equation*}
    \Ext^{p+2}_{D^e}(A,D^e) \; \cong \;  \Ext^p_{A^e}(A,{}_{\tau}A^e{}_{\tau^{-1}})
  \end{equation*}
  that preserves the right $D^e$-module structures. The result follows
  from the isomorphisms
  \begin{equation*}
    {}_\tau A_1\cong {}_1A_{\tau^{-1}},\quad {}_1 A_{\tau^{-1}}\cong {}_\tau A_1
  \end{equation*}
  for the left and respectively right-hand tensorands of $A^e$.
\end{proof}

\subsection{}
For the rest of this section we assume that $D$ is a connected graded algebra, $\Omega\in D_m$ is a normal regular element of degree $m$, and $A = D/(\Omega)$
is Artin-Schelter regular of dimension $d$. Since condition
\cref{item.enh.alg} of \Cref{pr.enh} holds, $D$ is Artin-Schelter  regular of dimension $d+1$.

By \Cref{lem.RRZ}, $A$ and $D$ are also twisted Calabi-Yau; it is the $A^e$- and $D^e$-bimodule structures of $A$ and $D$ that will play the central role.

We denote by $\nu=\nu_{\!{}_D}$ the Nakayama automorphism of $D$, defined by
\begin{equation}\label{eq:ddd}
  \Ext^{d+1}_{D^e}(D,D^e)\cong {}_\nu D_1 \cong {}_1D_{\nu^{-1}};
\end{equation}
all other $\Ext^p(D,D^e)$ vanish.

\subsection{}
The next result, together with \Cref{le.aa-quot}, ensures that
$\Ext_{D^{e}}^{d+2}(A,D^e)$ can function as a bridge between
$\Ext_{D^{e}}^{d+1}(D,D^e)$ and $\Ext_{A^{e}}^d(A,A^e)$, which in turn are used to
define the Nakayama automorphisms of $A$ and $D$.

\begin{lemma}\label{le.ddd-add}
  Let $D$ be a connected graded algebra, $\Omega\in D_m$ a normal regular element of degree $m$, and let $A = D/(\Omega)$.
If $A$ is AS-regular of dimension $d$, then there is an isomorphism
  \begin{equation*}
    \Ext_{D^e}^{d+2}(A,D^e) \; \cong \; {}_{\tau}A_{\nu_{\!{}_D}^{-1}} 
  \end{equation*}
  of $D$-bimodules.
\end{lemma}
\begin{proof}
	The short exact sequence \cref{eq:res-of-A.2} of $D$-bimodules yields an exact sequence
	\begin{equation*}
		\begin{tikzcd}
			\Ext_{D^e}^{d+1}(D,D^e)\arrow[r,"(\cdot\Omega)^{*}"] & \Ext_{D^e}^{d+1}({}_1D_\tau,D^e)\arrow[r] &
			\Ext_{D^e}^{d+2}(A,D^e)\arrow[r] & 0
		\end{tikzcd}
	\end{equation*}
	of right $D^e$-modules in which
	$(\cdot\Omega)^{*}$ denotes the morphism induced from the morphism $(\cdot\Omega)$ between the first variables. A morphism induced from that between second variables will be denoted by a lower star. We will show that the diagram 
	\begin{equation}\label{eq.le-ddd-add-1}
		\begin{tikzcd}
			\Ext_{D^{e}}^{d+1}(D,D^{e})
				\ar[d,"(\cdot\Omega)^{*}"]
				\ar[r,equal] &
			\Ext_{D^{e}}^{d+1}(D,D^{e})
				\ar[d,"((1\otimes\Omega)\cdot)_{*}"]
				\ar[r,equal] &
			\Ext_{D^{e}}^{d+1}(D,D^{e})
				\ar[d,"(\cdot(1\otimes\Omega))_{*}"]
				\ar[r,equal] &
			\Ext_{D^{e}}^{d+1}(D,D^{e})
				\ar[d,"\Omega\cdot"] \\
			\Ext_{D^{e}}^{d+1}({}_{1}D_{\tau^{}},D^{e})
				\ar[r,equal] &
			\Ext_{D^{e}}^{d+1}(D,{}_{1\otimes\tau^{-1}}D^{e})
				\ar[r,"\sim","(1\otimes\tau)_{*}"'] &
			\Ext_{D^{e}}^{d+1}(D,D^{e}_{1\otimes\tau})
				\ar[r,equal] &
			{}_{\tau}\Ext_{D^{e}}^{d+1}(D,D^{e})_{1}
		\end{tikzcd}
	\end{equation}
	commutes.
	The commutativity of the middle square follows from the definition of $\tau$. Note that right multiplication by $1\otimes\Omega$ on $D^{e}=D\otimes_{k}D^{\circ}$ is left multiplication by $\Omega$ on the right tensorand. The commutativity of the right-most square is straightforward.
	
	For the left-most square: The left vertical morphism is equal to $((1\otimes\Omega)\cdot)^{*}$. For each left $D^{e}$-module $M$, set $FM:={}_{1}M_{\tau^{-1}}$ and $\varphi_{M}:=((1\otimes\Omega)\cdot)\colon M\to FM$. This defines a morphism $\varphi\colon 1\to F$ of auto-equivalences on $\Mod(D^{e})$. Since every left $D^{e}$-homomorphism $f\colon FM\to N$ satisfies $F(f)F(\varphi_{M})=\varphi_{N}f$, we obtain the following commutative diagram by considering an injective resolution of $N$:
	\begin{equation*}
		\begin{tikzcd}
			\Ext_{D^{e}}^{d+1}(FM,N)
				\ar[d,"(\varphi_{M})^{*}"]
				\ar[r,equal] &
			\Ext_{D^{e}}^{d+1}(FM,N)
				\ar[d,"(\varphi_{N})_{*}"] \\
			\Ext_{D^{e}}^{d+1}(M,N)
				\ar[r,equal,""] &
			\Ext_{D^{e}}^{d+1}(FM,FN)\rlap{.}
		\end{tikzcd}
	\end{equation*}
	This becomes the left-most square of \cref{eq.le-ddd-add-1} after substituting $M:={}_{1}D_{\tau}$ and $N:=D^{e}$.
	
	In the following diagram, the right-hand square  obviously commutes and the left-hand square commutes since the horizontal arrows are isomorphism of $D$-bimodules:
	\begin{equation*}
		\begin{tikzcd}
			\Ext_{D^{e}}^{d+1}(D,D^{e})
				\ar[d,"\Omega\cdot"]
				\ar[r,"\sim"'] &
			{}_{\nu}D_{1}
				\ar[d,"\Omega\cdot"]
				\ar[r,"\nu^{-1}","\sim"'] &
			{}_{1}D_{\nu^{-1}}
				\ar[d,"\Omega\cdot"] \\
			{}_{\tau}\Ext_{D^{e}}^{d+1}(D,D^{e})_{1}
				\ar[r,"\sim"] &
			{}_{\tau}({}_{\nu}D_{1})_{1}
				\ar[r,"\nu^{-1}"',"\sim"] &
			{}_{\tau}D_{\nu^{-1}}
		\end{tikzcd}
	\end{equation*}
	where the middle vertical arrow is left multiplication by $\Omega$ through the twist by $\nu$, so as a linear map $D\to D$, it is left multiplication by $\nu(\Omega)$. By adjoining this diagram to the right-hand end of \cref{eq.le-ddd-add-1}, we obtain the commutative diagram
	\begin{equation*}
		\begin{tikzcd}
			\Ext_{D^e}^{d+1}(D,D^e)
				\ar[d,"\wr"]
				\ar[r,"(\cdot\Omega)^{*}"] &
			\Ext_{D^e}^{d+1}({}_1D_\tau,D^e)
				\ar[d,"\wr"]
				\ar[r] &
			\Ext_{D^e}^{d+2}(A,D^e)
				\ar[d,"\wr"]
				\ar[r] &
			0
			\\
			{}_{1}D_{\nu^{-1}}
				\ar[r,"\Omega\cdot"] &
			{}_{\tau}D_{\nu^{-1}}
				\ar[r] &
			{}_{\tau}A_{\nu_{\!{}_D}^{-1}}
				\ar[r] &
			0
		\end{tikzcd}
	\end{equation*}
	which completes the proof of the lemma.
\end{proof}

We now determine the relation between the Nakayama automorphisms of $D$ and $A$ when $D$ is generated by $D_1$ as a $\Bbbk$-algebra. 
The algebras $D$ that we will consider in \cref{sect.main.thm}  {\it are} generated by $D_1$.

\begin{theorem}\label{pr.nak-ad}
  Let $D$ be a connected graded $\Bbbk$-algebra that is generated by $D_1$.
  Let $\Omega\in D$ be a homogeneous normal regular element of degree $\geq 2$ and let $A = D/(\Omega)$.
If $A$ is twisted Calabi-Yau so is $D$ and in that case $\nu_{\!{}_A} = \tau^{-1}\nu_{\!{}_D}$
on the common space of generators $A_1=D_1$.
\end{theorem}
\begin{proof}
  It follows from \cref{item.enh.alg} of \Cref{pr.enh} that $D$ is twisted Calabi-Yau of
  dimension $d+1$.
  \Cref{le.aa-quot,le.ddd-add} together with $\Ext^d_{A^e}(A,A^e)\cong {}_{\nu_{\!{}_A}}A_1$ imply that
 \begin{equation*}
    {}_1A_{\nu_{\!{}_A}^{-1}} \; \cong \; \Ext^d_{A^e}(A,A^e) \; \cong \; {}_{\tau^{-1}}\Ext^{d+2}_{D^e}(A,D^e)_\tau  \; \cong \; {}_1 A_{\nu_{\!{}_D}^{-1}\tau};
  \end{equation*}
  the conclusions follow.
\end{proof}

We remind the reader that $\tau$ and $\nu_{\!{}_D}$ commute since, by \cite[Thm.~4.2]{lmz-nak}, Nakayama automorphisms
  are central under conditions that our algebras certainly satisfy.

\Cref{pr.nak-ad} is analogous to \cite[Lem.~1.5]{RRZ14}, but we do
not use the theory of dualizing complexes, and do not assume here that
$\Omega$ is an eigenvector for the Nakayama automorphism $\nu_{\!{}_D}$. In
fact, the latter condition now follows from the above discussion.

\begin{corollary}\label{cor.eigen}
  Let $D$ be a connected graded algebra, $\Omega\in D$ a homogeneous normal regular element of degree $\geq 2$. 
If $D/(\Omega)$ is twisted Calabi-Yau, then $\Omega$ is an eigenvector for the Nakayama automorphism $\nu_{\!{}_D}$.
\end{corollary}
\begin{proof}
By  \Cref{pr.nak-ad},  $\nu_{\!{}_D}$ descends to an automorphism of  $D/(\Omega)$. Hence $\nu_{\!{}_D}(\Omega)$ belongs to the ideal
  $(\Omega)$. Since $D$ is connected and $\nu_{\!{}_D}$ preserves degree, $\nu_{\!{}_D}(\Omega) \in \Bbbk \Omega$.
\end{proof}

\subsection{Homological determinant}
As before, let $D$ be a connected graded algebra, $\Omega\in D_m$ a normal regular element of degree $m$, and assume that $A = D/(\Omega)$ is 
Artin-Schelter regular of dimension $d$ with Gorenstein parameter $\ell$. 
We will determine the relation between the homological determinant of an automorphism of $D$ that preserves the ideal $(\Omega)$ with the 
homological determinant of the induced automorphism of $A$. 
The main result in this section, \cref{th.hdets}, will be used later to show that the homological determinant of the Nakayama automorphism
of the algebra $D(\sfw,\sfp)$ is 1.

The homological determinant was introduced by J{\o}rgensen-Zhang \cite{jz}. 
Before defining it we need some notation and other ideas.

Let $M$ be a left $A$-module.
The $p^{\rm th}$ {\sf local cohomology group} of $M$ is $H_{\fm}^{p}(M):=\varinjlim\Ext_{A}^{p}(A/A_{\geq k},M)$. 
The right action of $A$ on $A/A_{\geq k}$ determines a left $A$-module structure on 
$H_{\fm}^{p}(M)$, and $H_{\fm}^{p}$ then becomes an endofunctor on the category of graded left $A$-modules.

The {\sf Matlis dual} of $M$ is the graded right $A$-module  $M'$ whose degree-$i$ component is $\Hom_{\Bbbk}(M_{-i},\Bbbk)$, 
the right action  of $A$ coming from the left action of $A$ on $M$. 
The operation $M \rightsquigarrow M'$ is a duality between the category of graded left $A$-modules and the category of graded  right $A$-modules.

Let $\Aut_{\sf gr}(A)$ denote the group of degree-preserving $\Bbbk$-algebra automorphisms of $A$ and let $\sigma\in\Aut_{\sf gr}(A)$. 
Then $\s:A \to  {}_\s A_{\s}$ is an isomorphism of $A$-bimodules. 

J{\o}rgensen-Zhang \cite{jz} defined the homological determinant in the following way. First they show that the top local cohomology group 
$H_{\fm}^{d}(A)$ is isomorphic to $A'(\ell)$ and all other $H_{\fm}^{p}(A)$ vanish. Fix an isomorphism $\psi\colon H_{\fm}^{d}(A)\to A'(\ell)$. 
Let $H_{\fm}^{d}({}_\sigma A) \isoto {}_\sigma H_{\fm}^{d}(A)$ be the isomorphism 
\begin{equation*}
	H_{\fm}^{d}({}_\sigma A)\, =\, \varinjlim\Ext_{A}^{d}(A/A_{\geq k},{}_{\sigma}A)\, \isoto\, \varinjlim\Ext_{A}^{d}({}_{\sigma^{-1}}(A/A_{\geq k}),A)
	\, \isoto\, \varinjlim\Ext_{A}^{d}((A/A_{\geq k})_{\sigma},A)\,=\, {}_\sigma H_{\fm}^{d}(A).
\end{equation*}
The {\sf homological determinant} of $\s$ is the unique scalar $\hdet(\sigma)$ that makes the diagram 
\begin{equation}\label{eq.hdet}
	\begin{tikzcd}[column sep=large]
		H_{\fm}^{d}(A) \ar[r,"{H_{\fm}^{d}(\sigma)}","\sim"']\ar[d,"\psi"',"\wr"] & H_{\fm}^{d}({}_\sigma A)\ar[r,"\sim"'] & {}_\sigma H_{\fm}^{d}(A)\ar[d,"\psi","\wr"']\\
		A'(\ell)\ar[rr,"(\hdet(\sigma))^{-1}(\sigma^{-1})'"',"\sim"] & & {}_{\s}A'(\ell)
	\end{tikzcd}
\end{equation}
commute. When we need to identify the algebra $A$ we write $\hdet_{A}(\sigma)$.

The degree-$(-\ell)$ components of the diagram \cref{eq.hdet} give another commutative diagram:

\begin{equation}\label{eq.hdet2}
	\begin{tikzcd}[column sep=large]
		\Ext_{A}^{d}(\Bbbk,A)\ar[r,"\sigma_{*}","\sim"']\ar[d,"\wr"'] & \Ext_{A}^{d}(\Bbbk,{}_\sigma A)\ar[r,"{}_{\sigma^{-1}}(-)","\sim"'] & \Ext_{A}^{d}(\Bbbk,A)\ar[d,"\wr"]\\
		\Bbbk(\ell)\ar[rr,"(\hdet(\sigma))^{-1}"',"\sim"] & & \Bbbk(\ell)
	\end{tikzcd}
\end{equation}
where the two vertical morphisms are the same. 
Note that ${}_\sigma\Ext_{A}^{d}(\Bbbk,A)=\Ext_{A}^{d}(\Bbbk,A)$.
Because all the objects in \cref{eq.hdet2} have dimension 1,
the definition of $\hdet(\sigma)$ does not depend on the choice of the isomorphism $\psi$.

\begin{lemma}\label{lem.changebasering}
Let $\sigma \in \Aut_{\sf gr}(D)$.
If $\Omega$ is an eigenvector for $\sigma$, then
there is an isomorphism $\alpha\colon\Ext_{D}^{d}(\Bbbk,A)\isoto\Ext_{A}^{d}(\Bbbk,A)$ that makes the diagram 
	\begin{equation}\label{eq.changebasering}
		\begin{tikzcd}[column sep=large]
			\Ext_{D}^{d}(\Bbbk,A)\ar[r,"\sigma_{*}","\sim"']\ar[d,"\alpha"',"\wr"] & \Ext_{D}^{d}(\Bbbk,{}_\sigma A)\ar[r,"{}_{\sigma^{-1}}(-)","\sim"'] & \Ext_{D}^{d}(\Bbbk,A)\ar[d,"\alpha","\wr"'] \\
			\Ext_{A}^{d}(\Bbbk,A)\ar[r,"\sigma_{*}"',"\sim"] & \Ext_{A}^{d}(\Bbbk,{}_\sigma A)\ar[r,"{}_{\sigma^{-1}}(-)"',"\sim"] & \Ext_{A}^{d}(\Bbbk,A)
		\end{tikzcd}
	\end{equation}
commutative.
\end{lemma}

\begin{proof}
	Since $\sigma(\Bbbk\Omega)=\Bbbk\Omega$, $\sigma$ descends to an automorphism of $A$, which is also denoted by $\sigma$. By applying $\RHom_{D}(-,A)$ to the short exact sequence \cref{eq:res-of-A.2} of $D$-bimodules, we obtain the triangle in the first row of the diagram
	\begin{equation*}
		\begin{tikzcd}
			\RHom_{D}(A,A)\ar[r] & \RHom_{D}(D,A)\ar[r,"(\cdot\Omega)^{*}"] & \RHom_{D}({}_{1}D_{\tau},A)\ar[r] & \RHom_{D}(A,A)[1] \\
			\Hom_{D}(A,A)\ar[r]\ar[u] & \Hom_{D}(D,A)\ar[r,"(\cdot\Omega)^{*}"]\ar[u,"\wr"] & \Hom_{D}({}_{1}D_{\tau},A)\ar[u,"\wr"] \\
			A\ar[r,equal]\ar[u,"\wr"] & A\ar[r,"0"]\ar[u,"\wr"] & {}_{\tau}A\ar[u,"\wr"]
		\end{tikzcd}
	\end{equation*}
	in the left-bounded derived category of left $D$-modules. %Hence $\RHom_{D}(A,A)$ is, in the left-bounded derived category of left $A$-module, isomorphic to a two-term complex $C^{0}\to C^{1}$ whose cohomologies are $A$ at degree $0$ and ${}_{\tau}A$ at degree $1$. Since $\tau\colon A\to{}_{\tau}A$ is an isomorphism of left $A$-modules, ${}_{\tau}A$ is also projective. Therefore the two-term complex splits into $A\oplus {}_{\tau}A[-1]$.
	Since this diagram commutes, the first morphism in the triangle admits a section passing through $A$ in the bottom. Hence we have a split triangle
	\begin{equation}\label{eq.changebasering.split.tri}
		\begin{tikzcd}
			A\ar[r] & \RHom_{D}(A,A)\ar[r] & {}_{\tau}A[-1]\ar[r] & A[1]
		\end{tikzcd}
	\end{equation}
	where the first morphism is the composite of canonical morphisms $A\isoto\Hom_{D}(A,A)\to\RHom_{D}(A,A)$.
	
	Let $A\to I^{\hdot}$ be an injective resolution of $A$ as a left $D$-module. Then $\Hom_{D}(A,I^{\hdot})$ is a complex of injective left $A$-modules and we have
	\begin{equation*}
		A\isoto\Hom_{A}(A,A)=\Hom_{D}(A,A)\to\Hom_{D}(A,I^{\hdot}).
	\end{equation*}
	%whose $0^{th}$ cohomologies are isomorphic.
	This induces
	\begin{equation}\label{eq.changebasering.isom}
		\Ext_{A}^{d}(\Bbbk,A)\isoto\mathrm{H}^{d}\Hom_{A}(\Bbbk,\Hom_{D}(A,I^{\hdot}))\isoto\mathrm{H}^{d}(\Hom_{D}(\Bbbk,I^{\hdot}))=\Ext_{D}^{d}(\Bbbk,A)
	\end{equation}
	where the first isomorphism holds since we have \cref{eq.changebasering.split.tri} and $\mathrm{H}^{d}\Hom_{A}(\Bbbk,{}_{\tau}A[-1])\cong\Ext_{A}^{d-1}(\Bbbk,A)=0$ by the Gorenstein condition.
	
	The isomorphism $\Ext_{A}^{d}(\Bbbk,A)\isoto\Ext_{D}^{d}(\Bbbk,A)$ sends each $d$-extension of $A$ by $\Bbbk$ to itself. To see this, let 
	\begin{equation*}
		\xi\colon 0\to A\to E^{0}\to\cdots\to E^{d-1}\to\Bbbk\to 0
	\end{equation*}
	be a $d$-extension of left $A$-modules and let $A\to J^{\hdot}$ be an injective resolution of $A$ as a left $A$-module. Lifting up the canonical isomorphism $A\isoto\Hom_{D}(A,A)$, we obtain a commutative diagram
	\begin{equation*}
		\begin{tikzcd}
			0\ar[r] & A\ar[d,equal]\ar[r] & E^{0}\ar[d]\ar[r] & \cdots\ar[r] & E^{d-1}\ar[d]\ar[r] & \Bbbk\ar[d]\ar[r] & 0\rlap{.} \\
			0\ar[r] & A\ar[d,"\wr"]\ar[r] & J^{0}\ar[d]\ar[r] & \cdots\ar[r] & J^{d-1}\ar[d]\ar[r] & J^{d}\ar[d] \\
			0\ar[r] & \Hom_{D}(A,A)\ar[r] & \Hom_{D}(A,I^{0})\ar[r] & \cdots\ar[r] & \Hom_{D}(A,I^{d-1})\ar[r] & \Hom_{D}(A,I^{d})
		\end{tikzcd}
	\end{equation*}
	The image of $\xi$ by the first morphism in \cref{eq.changebasering.isom} is represented by the composite of the right-most vertical arrows. Applying the adjunction $(A\otimes_{A}-)\dashv\Hom_{D}(A,-)$ to the diagram, we obtain
	\begin{equation*}
		\begin{tikzcd}
			0\ar[r] & A\ar[d,equal]\ar[r] & E^{0}\ar[d]\ar[r] & \cdots\ar[r] & E^{d-1}\ar[d]\ar[r] & \Bbbk\ar[d]\ar[r] & 0\rlap{.} \\
			0\ar[r] & A\ar[r] & I^{0}\ar[r] & \cdots\ar[r] & I^{d-1}\ar[r] & I^{d}
		\end{tikzcd}
	\end{equation*}
	The right-most vertical morphism represents the image of $\xi$ by the composite \cref{eq.changebasering.isom}. Hence the image is the element of $\Ext_{D}^{d}(\Bbbk,A)$ represented by $\xi$ regarded as a $d$-extension of left $D$-modules.
	
	The commutativity of the diagram in the statement now follows from the above description using a $d$-extension.
\end{proof}

\begin{lemma}\label{lem.connectingmor}
Let $\sigma \in \Aut_{\sf gr}(D)$.
If $\Omega$ is an eigenvector for $\sigma$ with eigenvalue $\lambda$, then there is an isomorphism 
$\delta\colon\Ext_{D}^{d}(\Bbbk,A)\isoto\Ext_{D}^{d+1}(\Bbbk,D)$ that makes the diagram 
	\begin{equation}\label{eq.connectingmor}
		\begin{tikzcd}[column sep=large]
			\Ext_{D}^{d}(\Bbbk,A)\ar[r,"\sigma_{*}","\sim"']\ar[d,"\delta"',"\wr"] & \Ext_{D}^{d}(\Bbbk,{}_\sigma A)\ar[r,"{}_{\sigma^{-1}}(-)","\sim"'] & \Ext_{D}^{d}(\Bbbk,A)\ar[d,"\lambda^{-1}\delta","\wr"'] \\
			\Ext_{D}^{d+1}(\Bbbk,D)\ar[r,"\sigma_{*}"',"\sim"] & \Ext_{D}^{d+1}(\Bbbk,{}_\sigma D)\ar[r,"{}_{\sigma^{-1}}(-)"',"\sim"] & \Ext_{D}^{d+1}(\Bbbk,D)
		\end{tikzcd}
	\end{equation}
commutative.
\end{lemma}

\begin{proof}
	The short exact sequence \cref{eq:res-of-A.1} is extended to the commutative diagram in the derived category
	\begin{equation*}
		\begin{tikzcd}[column sep=huge]
			D\ar[r,"\cdot\Omega"]\ar[d,"\sigma","\wr"'] & D\ar[r]\ar[d,"\sigma","\wr"'] & A\ar[r,"\varepsilon"]\ar[d,"\sigma","\wr"'] & D[1]\ar[d,"{\sigma[1]}","\wr"'] \\
			{}_{\sigma}D\ar[r,"\cdot\sigma(\Omega)"]\ar[d,"\lambda","\wr"'] & {}_{\sigma}D\ar[r]\ar[d, equal] & {}_{\sigma}A\ar[r]\ar[d, equal] & {}_{\sigma}D[1]\ar[d,"\lambda","\wr"'] \\
			{}_{\sigma}D\ar[r,"(\cdot\Omega)={}_{\sigma}(\cdot\Omega)"] & {}_{\sigma}D\ar[r] & {}_{\sigma}A\ar[r,"{}_{\sigma}\varepsilon"] & {}_{\sigma}D[1]
		\end{tikzcd}
	\end{equation*}
	where each row is a triangle. The third and fourth columns yield the commutative diagram
	\begin{equation*}
		\begin{tikzcd}[column sep=huge]
			\Ext_{D}^{d}(\Bbbk,A)\ar[r,"\varepsilon_{*}","\sim"']\ar[d,"\sigma_{*}"',"\wr"] & \Ext_{D}^{d+1}(\Bbbk,D)\ar[d,"\lambda\sigma_{*}","\wr"'] \\
			\Ext_{D}^{d}(\Bbbk,{}_{\sigma}A)\ar[r,"({}_{\sigma}\varepsilon)_{*}","\sim"']\ar[d,"{}_{\sigma^{-1}}(-)"',"\wr"] & \Ext_{D}^{d+1}(\Bbbk,{}_{\sigma}D)\ar[d,"{}_{\sigma^{-1}}(-)","\wr"'] \\
			\Ext_{D}^{d}(\Bbbk,A)\ar[r,"\varepsilon_{*}","\sim"'] & \Ext_{D}^{d+1}(\Bbbk,D)
		\end{tikzcd}
	\end{equation*}
	Letting $\delta:=\Ext^{d}(\Bbbk,\varepsilon)$ completes the proof.
\end{proof}

\begin{theorem}\label{th.hdets}
  Let $D$ be a connected graded algebra, $\Omega\in D_m$ a normal regular element of degree $m$, and let $A = D/(\Omega)$.
Assume that $A$ is twisted Calabi-Yau of dimension $d$.  Let $\sigma \in \Aut_{\sf gr}(D)$.
If $\Omega\in D$ is an eigenvector for $\sigma$ with eigenvalue $\lambda$, then 
$$
\hdet_{D}(\sigma) \;=\; \lambda\hdet_{A}(\sigma|_{A}).
$$
\end{theorem}

\begin{proof}
	The concatenation of the diagram \cref{eq.hdet2} for $D$ with \cref{eq.changebasering} and \cref{eq.connectingmor} yields the commutative diagram
	\begin{equation*}
		\begin{tikzcd}[column sep=large]
			\Ext_{A}^{d}(\Bbbk,A)\ar[r,"\sigma_{*}","\sim"']\ar[d,"\wr"'] & \Ext_{A}^{d}(\Bbbk,{}_\sigma A)\ar[r,"{}_{\sigma^{-1}}(-)","\sim"'] & \Ext_{A}^{d}(\Bbbk,A)\ar[d,"\wr"] \\
			\Bbbk(\ell)\ar[rr,"(\hdet_{D}(\sigma))^{-1}"',"\sim"] & & \Bbbk(\ell)
		\end{tikzcd}
	\end{equation*}
	where the right isomorphism is equal to the left isomorphism multiplied by $\lambda^{-1}$. This shows $(\hdet_{A}(\sigma|_{A}))^{-1}=(\hdet_{D}(\sigma))^{-1}\lambda$.
\end{proof}

\subsection{An equivariant formulation of the homological determinant}

In this section we give a slightly different interpretation of the 
homological determinant. In this section $A$
denotes a  $d$-dimensional Artin-Schelter regular $\Bbbk$-algebra with
Gorenstein parameter $\ell$ and $\sigma$ is a
graded $\Bbbk$-algebra automorphism of $A$.

We assume some familiarity with group actions on categories and equivariantizations. 
A good reference  is \cite[$\S$2.7]{egno}. Additional details can be found in the appendix of  \cite{CS15}. 
 Since categories in this paper are $\Bbbk$-linear, all functors in the following discussion are required to  be $\Bbbk$-linear.

Recall \cite[Defns.~2.7.1 and 2.7.2]{egno}; because
 of the wide availability of resources on the topic we merely paraphrase the definitions here.

\begin{definition}\label{def.act}
  An {\sf action} of a group $\Gamma$ on a category $\cC$ is a
  collection of auto-equivalences $T_g:\cC \to \cC$, $g\in \Gamma$,
and natural isomorphisms $T_g\circ T_h\cong T_{gh}$ satisfying the
obvious compatibility conditions.

  Given an action of $\Gamma$ on $\cC$, the category $\cC^{\Gamma}$ of
  {\sf $\G$-equivariant objects}
  is the category whose objects are collections of isomorphisms 
  \begin{equation*}
    (\varphi_g:T_gx\to x \; | \; g \in \G)
  \end{equation*}
 that   are compatible with the endofunctors $T_g$
  in the sense that the diagrams 
  \begin{equation*}
    \begin{tikzpicture}[auto,baseline=(current  bounding  box.center)]
      \path[anchor=base] (0,0) node (tgh) {$T_{gh}x$} +(7.2,0) node (x) {$x$} +(2.4,-.5) node (tgth) {$T_gT_hx$} +(4.8,-.5) node (tg) {$T_gx$};

      \draw[->] (tgh) to[bend left=6] node[pos=.5,auto] {$\scriptstyle \varphi_{gh}$} (x);
      \draw[->] (tgh) to[bend right=6] node[pos=.5,auto,swap] {$\scriptstyle\cong$} (tgth);
      \draw[->] (tgth) to[bend right=4] node[pos=.5,auto,swap] {$\scriptstyle T_g\varphi_h$} (tg);
      \draw[->] (tg) to[bend right=6] node[pos=.5,auto,swap] {$\scriptstyle\varphi_g$} (x);
    \end{tikzpicture}
  \end{equation*}
  commute for all $g,h \in \G$, where the lower isomorphism is the   structural one defining the $\Gamma$-action on $\cC$.
  The morphisms in $\cC^\Gamma$ are those morphisms $f:x\to y$ in $\cC$ for which the diagrams
\begin{equation*}
  \begin{tikzcd}[column sep=huge]
    T_gx\ar[r,"T_gf"]\ar[d,"\varphi_g"'] & T_gy\ar[d,"\varphi_g"] \\
    x\ar[r,"f"'] & y
  \end{tikzcd}
\end{equation*}
commute for all $g\in \Gamma$.
\end{definition}

Now consider objects $x,y\in \cC^{\Gamma}$ (we are abusing notation by
suppressing the isomorphisms $\varphi_g:T_gx\to x$). The action of
$\G$ on $\cC$ induces an action $\triangleright$ of $\Gamma$ on
$\Hom_{\cC}(x,y)$ defined by the commutativity of the diagram
\begin{equation}\label{eq:conj-act}
  \begin{tikzcd}[column sep=huge]
    T_gx\ar[r,"T_g(g^{-1}\triangleright f)"]\ar[d,"\varphi_g"'] & T_gy\ar[d,"\varphi_g"] \\
    x\ar[r,"f"'] & y
  \end{tikzcd}
\end{equation}

\begin{remark}\label{re.inj-act}
  When $\cC$ is abelian a $\Gamma$-action passes over to any of the
  derived categories associated to $\cC$ (the bounded derived category
  $\cD^b(\cC)$, the left bounded one $\cD^+(\cC)$, etc.). We assume below that $\cC$ has enough injectives.
  
  For an equivariant object $y$ in $\cC^\Gamma$ as above the
  isomorphisms $\varphi_g:T_gy\to y$ lift to isomorphisms between the
  injective resolutions $T_gE^\hdot$ and $E^\hdot$ of $T_gy$ and $y$
  respectively, and given another equivariant object $x\in \cC$ the
  action of $\Gamma$ on
  \begin{equation*}
    \Ext^d_{\cC}(x,y) = \Hom_\cD(x,y[d]) 
  \end{equation*}
  can be computed via the action on $\Hom_\cC(x,E^d)$ as in
  \Cref{eq:conj-act}, using the isomorphism $\varphi_g:T_gx\to x$ as the
  left-hand vertical arrow and the isomorphism $T_gE^d\to E^d$ as the
  right-hand vertical map.
  
All group actions on derived categories considered below arise in this
fashion, from an action on the original abelian category, and
similarly, we only consider equivariant objects in $\cD^b(\cC)$ that
are lifts of those in $\cC$ in the manner described above.
\end{remark}

We now apply the above to the group
$\Gamma=\bZ$ acting via the twist $M\mapsto {}_{\sigma}M$ for $M$ in
the category of $A$-modules, graded modules, complexes thereof,
derived categories, etc. Since the group $\Gamma$ is generated by the
single element $g$ mapping to $\sigma\in \langle \sigma\rangle$,
equivariant structures are determined by single automorphisms
$\varphi^x=\varphi_g:T_gx\to x$.

By definition, for (graded) $A$-modules the twist functor 
\begin{equation*}
  M \; \mapsto \; T_gM \; =\; {}_{\sigma}M
\end{equation*}
does not change the underlying vector spaces and morphisms, so the  diagram \Cref{eq:conj-act} above, for equivariant objects $(M,\varphi^M)$ and
$(N,\varphi^N)$, reads
\begin{equation}\label{eq:conj-act-concr}
  \begin{tikzcd}[column sep=huge]
    {}_{\sigma}M\ar[r,"g^{-1}\triangleright f"]\ar[d,"\varphi^M"'] & {}_{\sigma}N\ar[d,"\varphi^N"] \\
    M\ar[r,"f"'] & N
  \end{tikzcd}
\end{equation}

We can apply the discussion above to the bounded derived category
$\cD^b=\cD^b(\Mod(A))$ of left  $A$-modules with $\Gamma=\bZ$ once more
acting by powers by $\sigma$.

Now regard all $A/A_{\ge k}$, $k\le \infty$ (where $A/A_{\ge\infty}$ means $A$), and their homological shifts $(A/A_{\ge k})[p]$ in
$\cD^b$ as $\sigma$-equivariant $A$-modules  via
$\varphi^{(A/A_{\ge k})[p]}:=\sigma^{-1}[p]$. This renders the following
statement meaningful; it is a repackaging of \cite[Defn.~2.3]{jz}.

\begin{proposition}\label{le.hdet-alt}
  The inverse $\hdet^{-1}:=\hdet(\sigma)^{-1}$ of the homological determinant is
  the scalar by which $g^{-1}\triangleright$   acts on the
  degree-$(-\ell)$ component of any of the spaces $H_{\fm}^d(A)$,
  $\Ext_A^d(\Bbbk,A)$, or $\Ext_A^d(\Bbbk,\Bbbk)$. 
\end{proposition}
\begin{proof}
  First, note that the three scalars in the statement are indeed
  equal:

  As far as the last two are concerned, the long exact sequence
  attached to the short exact sequence
  \begin{equation*}
    0\to A_{\ge 1}\to A\to\Bbbk\to 0,
  \end{equation*}
  the vanishing of $\Ext_A^{d+1}$ and the fact that $\Ext_A^d(\Bbbk,A)$
  and $\Ext_A^d(\Bbbk,\Bbbk)$ are both one-dimensional vector spaces jointly
  ensure that the induced map
  \begin{equation*}
    \Ext_A^d(\Bbbk,A)\cong \Ext_A^d(\Bbbk,\Bbbk)
  \end{equation*}
  is an isomorphism.

  On the other hand, the first two scalars in the statement are equal
  because the $\sigma$-equivariant maps $A/A_{\ge k+1}\to A/A_{\ge k}$
  induce a $\Gamma$-module morphism
  \begin{equation*}
    \Ext_A^d(\Bbbk,A)\to H_{\fm}^d(A) = \varinjlim \Ext_A^d(A/A_{\ge k},A)
  \end{equation*}
  that identifies degree-$(-\ell)$ components.

Thus, it suffices to prove that the scalar $\hdet^{-1}$
  defined via \Cref{eq.hdet2} agrees with the eigenvalue of $g^{-1}$
  on $\Ext_A^d(\Bbbk,A)$.
  %(note however that for us, here, the top
  %dimension is $d$ rather than the $d+1$ of \Cref{eq.hdet2}). 
  We proceed to do this.

  \cite[Defn.~2.3]{jz} (which \Cref{eq.hdet2} follows) dictates
  that the scalar $\hdet^{-1}$ is computed as follows:

  Consider the minimal injective resolution
  \begin{equation*}
    0\to A\to E^0\to E^1\to\cdots
  \end{equation*}
  of $A$ in $\Mod(A)$, and let $\sigma$ act on it so as to extend the
  action on $A$. Then, $\hdet^{-1}$ will be the scaling that $\sigma$
  induces on the one-dimensional torsion of the $d^{th}$ term
  $E^d$. This is also what we obtain by making $g$ (the generator of
  $\Gamma=\bZ$) act on $\Hom(\Bbbk,E^d)$ by sending a vector space
  morphism $f:\Bbbk\to E^d$ to $\sigma\circ f\circ \sigma^{-1}$, which is
  what \Cref{eq:conj-act-concr} amounts to in the present setting (see \Cref{re.inj-act}).
\end{proof}

%%%%%%%%%%%%%%%%%%%%%%%%%%%%%%%%%%%%%%%%%%%%%%%%%%%%%%%%%%%%%%%%%%%%%%%%%%%%%%%%%
%%%%%%%%%%%%%%%%%%%%%%%%%%%%%%%%%%%%%%%%%%%%%%%%%%%%%%%%%%%%%%%%%%%%%%%%%%%%%%%%%
\section{Normal extensions of $3$-dimensional regular algebras}
\label{sect.main.thm}

In this section we describe a technique for producing $4$-dimensional Artin-Schelter regular algebras by modifying
$m$-Koszul $3$-dimensional Artin-Schelter regular algebras.  
Again, our regular algebras need not be noetherian.
We first recall some conventions and
terminology from various sources on $m$-Koszul and superpotential
algebras, such as \cite{dv1,dv2,bsw,MS16}.

\subsection{Superpotential algebras}
\label{subsec.superpotential}
For our purposes, the relevant setup is as follows. 

We fix an $n$-dimensional vector space $V$ that will be the common space of
degree-one generators for the algebras under discussion.
We also fix  integers $\ell \ge m \ge 2$, an invertible linear transformation $Q\in {\rm GL}(V)$,   and an element
${\sf w}\in V^{\otimes \ell}$ that is invariant under the linear map
\begin{equation*}
V^{\otimes \ell} \longrightarrow  V^{\otimes \ell}, \quad v_1\otimes \ldots \otimes v_\ell \mapsto (Q^{-1}v_\ell) \otimes v_1 \otimes \ldots \otimes v_{\ell-1}.
\end{equation*}
We call such a $\sfw$ a {\sf $Q$-twisted superpotential} or just a {\sf twisted superpotential} if we don't want to specify $Q$.
%This should be compared to \cite[(2.12)]{AS87} and the terminology in \cite[$\S$2.2]{bsw}.
If $Q=\mathrm{id}_{V}$ we simply call ${\sf w}$ a {\sf superpotential}.%; see also \cite[Defn.~2.1]{mu-cy}. 

Frequently, we assume that $Q$ is diagonalizable, but see (2) in \S\ref{re.gen}.

Let   $W$ be a subspace of some tensor power,  $V^{\otimes p}$ say.
We introduce the following notation:
\begin{align*}
\pd W  & \; :=\; \big\{(\psi \otimes \id^{\otimes p-1})(\sfw) \; \big\vert \; \psi \in V^* \; \text{and} \; \sfw \in W\big\},
\\
\pd^{i+1}W  &\; :=\;  \pd(\pd^i W) \quad \hbox{for all $i \ge 0$,}
\\
A(W, i)  &\; :=\; TV/(\partial^iW).
\end{align*}
The space $\pd^iW$ appears in \cite[\S4]{dv2} where it is denoted $W^{(p-i)}$. 

Let $\sfw\in V^{\otimes \ell}$ be a $Q$-twisted superpotential.  We call the algebra
$$
A(\sfw, i) \; :=\;  A(\Bbbk\sfw, i)
$$
a  {\sf superpotential algebra}.
In \cite{bsw}, $A(\sfw, i)$ is called a derivation-quotient algebra.
We are mostly interested in the case $i=1$ and we then write
$$
A(\sfw) \; := \; A(\sfw,1).
$$

\begin{theorem}
[Dubois-Violette]
 \textnormal{\cite[Thm.~11]{dv2}}
 \cite[Thm.~6.8]{bsw}
\label{thm.DV} 
If $A$ is an $m$-Koszul Artin-Schelter regular algebra of dimension $d$ with Gorenstein parameter $\ell$, then there is a twisted superpotential
$\sfw\in V^{\otimes \ell}$ such that  
$$
A \;  = \; A(\sfw, \ell-m)
$$
where such $\sfw$ is unique up to non-zero scalar multiple.

The global dimension $d$ of $A({\sfw},\ell-m)$ can be recovered
from the other numerical data as follows:
\begin{itemize}
\item if $m=2$, then $d=\ell$;
\item if $m\ge 3$, then $d$ is odd and $\ell=m\frac{d-1}{2}+1$.
\end{itemize}
\end{theorem} 

We will use the next lemma in \cref{cor.cy2}.

\begin{lemma}\label{le.hdet}
Let $k \in \ZZ$.
If $\sfw \in V^{\otimes \ell}$ is a $Q$-twisted superpotential, % such that $A({\sf w},\ell-m)$ is an $m$-Koszul twisted Calabi-Yau algebra, 
then
\begin{equation*}
(Q^k)^{\otimes\ell} {\sf w} \; =\;  {\sf w}.
\end{equation*} 
\end{lemma}
\begin{proof}
The $\ell^{\rm th}$ power of the linear operator on $V^{\otimes\ell}$ in the definition of a $Q$-twisted superpotential is $(Q^{-1})^{\otimes\ell}$. Since $\sfw$ is invariant under the operator, the statement follows.
%  This follows from the conjunction of \cite[Thms.~1.2, 1.6, and   1.8]{MS16}.
\end{proof}

\begin{remark}
Our $Q$-twisted superpotential is a $(Q^{-1})$-superpotential in the sense of \cite[Defn.~2.5]{MS16} and \cite[Defn.~2.1]{mu-cy}. In \cite[\S 2.2]{bsw}, it was called a twisted superpotential of degree $\ell$ with respect to $(-1)^{\ell-1}Q^{-1}$. Our superpotentials can be related to potentials of a quiver in \cite[\S 3]{dwz} (also to superpotentials in \cite[\S 1]{Boc08}) by considering a quiver $\Gamma$ consisting of one vertex and $\dim V$ loops. The space of potentials of $\Gamma$ concentrated in length $\ell$ modulo cyclic equivalence in the sense of \cite[\S 3]{dwz} has a basis consisting of cycles modulo cyclic permutation. So the space is isomorphic to the space of superpotentials in $V^{\otimes\ell}$ in our sense.
\end{remark}

\subsection{Notation}
\label{subsec.not.cycder}
It will be useful in what follows to use a fixed basis $x_1,\ldots,x_n$ for $V$. The matrix representing the linear transformation $Q$ is also denoted by $Q$. That is, $Q(x_{j})=\sum_{i}Q_{ij}x_{i}$.

For each $i=1,\ldots,n$, let $\partial_i:TV\to TV$ be the
linear map that acts on words ${\sf u}$ in the letters $x_j$ as
follows:
$$
\partial_i{\sf u} \;=\; 
\begin{cases}
\sfv & \text{if $\sfu=x_i\sfv$,}
\\
0 & \text{otherwise.}
\end{cases}
$$
We call $\pd_i\sfu$ the {\sf partial derivative} of $\sfu$ with
respect to $x_i$.  It is easy to see that $A({\sf w},\ell - m)$ is
equal to the tensor algebra $TV$ modulo the ideal generated by
$$\{\partial_{i_1}\ldots \partial_{i_{\ell-m}}(\sfw) \; | \; 1\le i_1,\ldots,i_{\ell-m} \le n\}.
$$

From now on $A$ will denote a 3-dimensional $m$-Koszul Artin-Schelter
regular algebra so, by \cref{thm.DV}, $\ell=m+1$ and there is a twisted superpotential $\sfw\in V^{\otimes (m+1)}$ such that
$$
A  \; = \; A(\sfw) \;=\; \frac{TV}{(\pd_1\sfw, \ldots,\pd_n\sfw)}.
$$
The minimal projective resolution of ${}_A\Bbbk$ is of the form
$0\to A(-\ell) \to A \otimes R \to A \otimes V \to A \to \Bbbk \to 0$
where $R$ is the linear span of a ``minimal'' set of relations. The
Gorenstein condition implies that $\dim(V)=\dim(R)$, so
$\{\pd_1\sfw, \ldots,\pd_n\sfw\}$ is linearly independent in
$TV$. This follows also from the ``pre-regularity'' of $\sfw$ in
\cite[Thm.~11]{dv2}.

If $p \in \Bbbk$ and $x$ and $y$ are elements in a $\Bbbk$-algebra, we use the
notation
\begin{equation*}
  [x,y]_p  \; := \;  xy-pyx. 
\end{equation*}

\subsection{The main theorem}
\label{def.stg}

Let $ \sfw \in V^{\otimes (m+1)}$ be a $Q$-twisted superpotential and
let $\sfp=(p_1,\ldots, p_{n})$ be a tuple of non-zero scalars.  The
algebra $D(\sfw,\sfp)$ is the tensor algebra $TV$ modulo the relations
\begin{equation}
  \label{eq:d-rels}
  \partial_i{\sfw} \; = \; [x_i,\partial_1{\sfw}]_{p_i} \;=\; 0  \qquad \text{ for }i=2,\ldots,n.
\end{equation}
Thus $D(\sfw,\sfp)$ has $n-1$ relations in degree $m$ and $n-1$
relations of degree $m+1$.  The scalar $p_{1}$ is not used in the
relation and will later be fixed to be a particular entry in $Q$.
 
We will frequently write $A$ for $A({\sfw})$ and $D(\sfp)$, or just
$D$, for $D(\sfw,\sfp)$.

We will write $\Omega$ for $\pd_1\sfw$ viewed as an element in
$D(\sfw,\sfp)$.  Clearly, $A(\sfw) = D(\sfw,\sfp)/(\Omega)$.

The following notation will be used in what follows. First, we write
$$
\sff  \; =\; \begin{pmatrix} f_1 \\ \vdots \\  f_n \end{pmatrix} \; =\; \begin{pmatrix} \pd_1\sfw \\ \vdots \\ \pd_n \sfw \end{pmatrix}
\qquad \text{and} \qquad 
\sfx \; =\; \begin{pmatrix} x_1 \\ \vdots \\  x_n \end{pmatrix}.
$$
Since $\sfw \in V^{\otimes(m+1)}$, there is a unique $n \times n$ matrix $\sM$ with entries in $V^{\otimes(m-1)}$ such that 
$$
\sfw \;=\; \sfx^t \sfM \sfx.
$$
It follows from the definition of the partial derivatives $\pd_j\sfw$ that 
\begin{equation*}
  \sfM{\sf x} \; =\;  {\sff}.
\end{equation*}
Because $\sfw$ is a $Q$-twisted superpotential,
\begin{equation}\label{eq.tw.sp}
{\sf x}^t \sfM \; =\;  (Q{\sff})^t
\end{equation}
where $Q$ is regarded as a matrix as explained in \cref{subsec.not.cycder}, and the products are matrix products whose entries are in $TV$. The equation \cref{eq.tw.sp} is written in \cite[(2.1)]{AS87} and the equivalence to the definition of a superpotential follows from \cite[(2.12)]{AS87}. Note that they used the dual space $V^{*}$ in the place of our $V$ and hence the matrices representing linear transformations on $V$ are transposed from ours.
%The last two equalities take place in $TV$.
%We are following the notation in \cite{AS87}; see, for example,  \cite[(2.1)]{AS87}. 

\begin{lemma}\label{le.stg-cntr}
If $Q$ is diagonal with entries $q_1,\ldots,q_n$ with respect
  to the basis $x_1,\ldots,x_n$ for $V$, then $\Omega$ is a normal
  element and $A({\sfw})=D(\sfw,\sfp)/(\Omega)$.
\end{lemma}
\begin{proof}
Since $[x_j,\partial_1{\sfw}]_{p_j}=0$ in $D(\sfw,\sfp)$, 
to prove $\Omega$ is normal it only remains to show that $x_1\Omega$ is a non-zero scalar multiple of $\Omega x_1$.
It follows from the  remarks prior to this lemma that  there are equalities
\begin{equation*}
\sum_{i=1}^n x_i f_i \;=\;  {\sfx}^t \sff \;=\;  {\sfx}^t  \sfM \sfx \;=\;   (Q{\sff})^t \sfx  \;=\;  \sff^t Q^t{\sfx} \;=\;  \sff^t Q{\sfx} \;=\; \sum_{i=1}^n q_i f_i x_i,
\end{equation*}
in $TV$.   But $f_i=0$ in $D(\sfw,\sfp)$ for all $i \ge 2$, so $x_1\Omega = q_1 \Omega x_1$. 
Thus, $\Omega$ is normal as claimed.
\end{proof}

Now we state our main theorem.

\begin{theorem}\label{th.3-stg}
  Let $A({\sfw})$ be an $m$-Koszul 3-dimensional Artin-Schelter regular algebra generated by $x_1,\ldots,x_n$.
  Let $\sfp=(p_1,\ldots,p_{n})$ be a tuple of non-zero scalars  and define $D(\sfw,\sfp)$ as above.
If $Q$ is a diagonal matrix $\diag(q_{1},\ldots,q_{n})$ with respect to   $x_1,\ldots,x_n$ and $p_1=q_1$, then
the following conditions are equivalent:
 \begin{enumerate}
  \item\label{item.reg} $D(\sfw,\sfp)$ is a 4-dimensional Artin-Schelter regular algebra;
  \item\label{eq:sk}
    $[\Omega,\sfM_{ij}]_{q_1^{-1}p_ip_j}=0$ in $D(\sfw,\sfp)$ for all
    $1\le i,j\le n$;
  \item\label{item.param}
    for all $1\leq i,j\leq d$ and all words
    $x_{l_{1}}\ldots x_{l_{m-1}}$ appearing in $\sfM_{ij}$,
    $q_{1}=p_{i}p_{j}p_{l_{1}}\ldots p_{l_{m-1}}$;
  \item\label{item.eigenvec} the automorphism $\varphi\colon V\to V$ defined by $\varphi(x_{i})=p_{i}x_{i}$ has the property
  \begin{equation*}
    \varphi^{\otimes(m+1)}(\sfw)=q_{1}\sfw.
  \end{equation*}
\end{enumerate}
Furthermore, if one of these conditions holds, then $\Omega$ is a regular element of $D(\sfw,\sfp)$.
\end{theorem}

We will prove the theorem in \cref{sec.proof}.

\subsection{Remarks}
\label{re.gen}
(1) We say the tuple $\sfp$ is {\sf good} if it satisfies one of the
conditions \cref{eq:sk}-\cref{item.eigenvec} in \Cref{th.3-stg}.

(2) \Cref{th.3-stg} can be generalized and extended in a number of
ways that we will take for granted below.  For example, instead of
replacing the relation $\pd_1\sfw=0$ we could replace the relation
$\partial_k{\sfw}=0$ for {\it any} $k \in \{1,\ldots,n\}$; the proof
proceeds just as for $k=1$ after making the obvious changes to the
indexing.  Furthermore, it is not necessary that $Q$ be diagonal; as
we note below in \Cref{re.blk}, it is enough to assume that the entry
$q_k$ relevant to the relation $\partial_k{\sfw}$ being omitted splits
off as a block diagonal entry.

(3)
If one of the conditions in \Cref{th.3-stg} holds, then $D(\sfw,\sfp)$ is a domain if $A(\sfw)$ is. The argument is well-known.
Let $R$ be an $\NN$-graded ring and $\Omega$ a homogeneous regular normal element of positive degree. Suppose $R/(\Omega)$ is 
a domain. If $R$ is not a domain there are non-zero homogeneous elements $a$ and $b$ in 
$R$ such that $ab=0$; among all such pairs $(a,b)$ choose one such that $\deg(a)+\deg(b)$ is as small as possible;\
since  $R/(\Omega)$ is a domain either $a$ or $b$ is a multiple of $\Omega$; without loss of generality, assume $a=\Omega a'$; 
hence $\Omega a'b=0$; since $\Omega$ is regular, $a'b=0$; but   $\deg(a')+\deg(b) < \deg(a)+\deg(b)$ which contradicts the choice of $a$ and $b$ so we conclude that $R$ is, in fact, a domain.

%%%%%%%%%%%%%%%%%%%%%%%%%%%%%%%%%%%%%%%%%%%%%%%%%%%%%%%%%%%%%%%%%%%%%%%%%%%%%%%%%

\subsection{Consequences of the main theorem}

\Cref{th.3-stg} is particularly useful when $Q=\id_V$, i.e., when $\sfw$ is
a superpotential. In our setting, it follows from \cite[Cor.~4.5]{MS16} that $\sfw$
is a superpotential if and only if $A(\sfw)$ is a $3$-Calabi-Yau algebra.
In that case, $\sfp=(1,\ldots,1)$ is good because it
obviously satisfies condition (3), and equally obviously satisfies
condition (4), in \cref{th.3-stg}. We record this observation since we
will use it later.

\begin{corollary}\label{cor.cy}
  Let $A=A({\sfw})$ be a $3$-Calabi-Yau algebra generated by $x_1,\ldots,x_n$.
 Then the algebra $D$ generated by $x_1,\ldots,x_n$ subject to  relations
  \begin{equation*}
    \pd_i\sfw \;=\;  [x_i,\partial_1{\sfw}] \;=\; 0 \qquad (i=2,\ldots,n)
  \end{equation*}
  is a $4$-dimensional Artin-Schelter regular algebra and is a central
  extension of $A$ by $\Omega=\partial_1{\sfw}$.
\end{corollary}

There is a more general version of  \Cref{cor.cy}.

\begin{corollary}\label{cor.alt-cy}\label{cor.cy2}
  Let $k\in \bZ$.  Let $A({\sfw})$ be $3$-dimensional Artin-Schelter
  regular (= twisted Calabi-Yau) algebra generated by
  $x_1,\ldots,x_n$.  If $Q=\diag(1,q_{2},\ldots,q_{n})$ with respect
  to the basis $x_1,\ldots,x_n$, then the algebra $D$ generated by
  $x_1,\ldots,x_n$ subject to the relations
  \begin{equation*}
    \pd_i\sfw \;=\;  [x_i,\partial_1{\sfw}]_{q_i^k} \;=\; 0 \qquad (i=2,\ldots,n)
  \end{equation*}
  is a $4$-dimensional Artin-Schelter regular algebra and is a normal
  extension of $A(\sfw)$ by $\Omega=\partial_1{\sfw}$.
\end{corollary}
\begin{proof}
  It suffices to show that $\sfp=(1,q_2^k,\ldots,q_n^k)$ is good.  By
  \Cref{le.hdet}, for every monomial $x_{l_{1}}\ldots x_{l_{m+1}}$
  appearing in ${\sf w}$ the product $\prod_{i=1}^{m+1}q_{l_i}^k$
  equals $1$.  Hence $\sfp$ satisfies the criterion for goodness in
  \cref{th.3-stg}\cref{item.eigenvec}.
\end{proof}

In \cref{cor.d-cy} we will show that the algebra $D$ in \cref{cor.cy}
is $4$-Calabi-Yau. The next result shows that every $4$-Calabi-Yau
central extension of $A(\sfw)$ can be obtained in this way.

\begin{theorem}\label{th.allcyext}
  Let $A(\sfw)$ be a $m$-Koszul $3$-Calabi-Yau algebra. Then 
  \emph{all} central extensions of
  $A(\sfw)$ by a degree-$m$ element can be obtained by the
  construction in \cref{cor.cy}. Consequently, they are $4$-Calabi-Yau.
\end{theorem}

\begin{proof}
  Let $D'$ be a central extension of $A(\sfw)$. Explicitly, let
  $\Omega\in D'$ be a degree-$m$ central regular element such that
  $D'/(\Omega)\cong A(\sfw)$.  Note that $D'_{1}=A_{1}=V$. Since the
  relations for $A$ are concentrated in degree $m$ and span an
  $n$-dimensional subspace of $TV$ (see \cref{subsec.not.cycder}),
  there exist degree-$m$ elements $f_{1},\ldots,f_{n}$ of $TV$ such
  that $A=TV/(f_{1},\ldots,f_{n})$, and $f_{2}=\cdots=f_{n}=0$ in $D'$,
  and $\Omega$ is the image of $f_{1}$ in $D'$.

  Let $y_{1},\ldots,y_{n}$ be a basis of $V$ and write
  $\sfw=y_{1}h_{1}+\cdots+y_{n}h_{n}$. Since $A(\sfw)$ is defined by
  the relations $h_{1},\ldots,h_{n}$, there is a matrix $P\in\GL(n)$ such
  that
  $(h_{1},\ldots,h_{n})^{t}=P(f_{1},\ldots,f_{n})^{t}$. Define
  the basis $x_{1},\ldots,x_{n}$ of $V$ by
  $(x_{1},\ldots,x_{n})=(y_{1},\ldots,y_{n})P$. Then
\begin{equation*}
	\sfw=(y_{1},\ldots,y_{n})\begin{pmatrix}h_{1} \\ \vdots \\ h_{n}\end{pmatrix}=(y_{1},\ldots,y_{n})P\begin{pmatrix}f_{1} \\ \vdots \\ f_{n}\end{pmatrix}=(x_{1},\ldots,x_{n})\begin{pmatrix}f_{1} \\ \vdots \\ f_{n}\end{pmatrix}.
\end{equation*}
Hence $f_{i}=\pd_{i}\sfw$ where the partial derivatives are taken with
respect to the basis $x_{1},\ldots,x_{n}$. Since $\Omega\in D'$ is
central, the defining relations for the algebra $D$ in \cref{cor.cy}
are satisfied by $D'$. Hence there is a surjective homomorphism
$D \to D'$. But $D$ and $D'$ have the same Hilbert series because $A$
is a quotient of each of them by a central regular element of degree
$m$. Thus $D'=D$.
\end{proof}

\subsection{A flat family}\label{subse.fam}
The other consequence of  \Cref{th.3-stg}  we wish to highlight is \cref{thm2} in the introduction. 
That result says that if $A(\sfw)$ is a 3-Calabi-Yau algebra on $n$ generators, then the recipe in \Cref{cor.cy} 
produces a flat family over  $\PP^{n-1}$  of 4-Calabi-Yau algebras $D_x$, $x \in V-\{0\}$,  that are central extensions of $A(\sfw)$.

In the statement of \cref{thm2}, we defined $D_x:=TV/I_x$ where $I_x$
is the ideal generated by
\begin{equation}
  \label{eq:rel1}
\langle x^\perp,\sfw \rangle \;=\; \{\langle \psi,\sfw\rangle \; | \; \psi \in x^\perp\}  
\end{equation}
and
\begin{equation}
  \label{eq:rel2}
[V,\langle V^*,\sfw \rangle]  \;=\; \{[y,\langle \psi,\sfw\rangle] \; | \; y \in V, \, \psi \in V^*\}.
\end{equation}
If $\l \in \Bbbk-\{0\}$, then $D_{\l x}=D_x$ so we will often write
$x$ for the point in $\PP(V)$ that corresponds to the line $\Bbbk x$ in
$\PP(V)$. Thus, we have an algebra $D_x$ for each $\Bbbk$-point
$x \in \PP(V)$.

Before showing that the $D_x$'s provide a flat family we check they have the other properties we want.

\begin{lemma}
 Let $A({\sfw})$ be a $3$-Calabi-Yau algebra generated by $x_1,\ldots,x_n$.
\begin{enumerate}
  \item\label{item.def.D} 
  The algebra $D_{x_1}$ is the same as the algebra $D$ constructed in  \Cref{cor.cy}.
  \item\label{item.fam.cent.ext} 
  The algebras $D_x$ are 4-Calabi-Yau algebras having central  regular elements $\Omega_x$ 
such that $D_x/(\Omega_x)=A(\sfw)$. 
\end{enumerate}
\end{lemma}
\begin{proof}
\cref{item.def.D}
Let $D$ be the algebra constructed in  \Cref{cor.cy}.

The hypothesis that $\sfw$ is a superpotential
  says there are elements $\sfw_1,\ldots,\sfw_n \in V^{\otimes m}$
  such that
  \begin{equation*}
     \sfw \;=\; x_1\sfw_1+\cdots +x_n\sfw_n \;=\; \sfw_1x_1+\cdots +\sfw_nx_n.
  \end{equation*}

By definition, $D=TV/I$ where $I$ is the ideal generated by
$\{\sfw_2,\ldots,\sfw_n\}$ and
$\{[x_i,\sfw_1] \; | \; i=2,\ldots,n\}$.  This is clearly the same as
the ideal generated by $\{\sfw_2,\ldots,\sfw_n\}$ and
$\{[x_i,\sfw_j] \; | \; 1 \le i,j \le n\}$.  
We note that
$$
{\rm span}\{\sfw_1,\ldots,\sfw_n\} \;=\;  \{ \langle \psi, \sfw\rangle \; | \; \psi \in V^*\}\;= \;\langle V^*,\sfw \rangle
$$
and
$$
{\rm span}\{\sfw_2,\ldots,\sfw_n\} \;=\;  \{ \langle \psi, \sfw\rangle \; | \; \psi \in x_1^\perp\} \;=\; \langle x_1^\perp,\sfw \rangle.
$$
It is now apparent that $I=I_{x_1}$, i.e., $D=D_{x_1}$.

\cref{item.fam.cent.ext}
Since $\dim_\Bbbk(\langle V^*,\sfw \rangle)=n$, $\langle x^\perp,\sfw \rangle$ is a codimension-one subspace of $\langle V^*,\sfw \rangle$
or, equivalently, the image of $\langle V^*,\sfw \rangle$ in $TV/I_x$ has dimension 1; thus, if $\Omega_x$ is a non-zero element in it, then
$D_x/(\Omega_x)=A(\sfw)$. 
In $D_x$, $[y,\langle \psi,\sfw\rangle] =0$ for all $y \in V$ and all $\psi \in V^*$. This says that $\Omega_x$ belongs to the center of 
$D_x$.
\end{proof}

\begin{remark}
  It is reasonable to ask whether as $x$ varies over $V-\{0\}$
  every codimension-1 subspace of the degree-$m$ relations for
  $A(\sfw)$ arises as the degree-$m$ relations for some $D_x$. The
  answer is ``yes'' and the proof is as follows.  Let
  $R'={\rm span}\{\sfw_2',\ldots,\sfw_n'\}$ be a codimension-1
  subspace of $\langle V^*,\sfw \rangle$. Let
  $\sfw_1' \in \langle V^*,\sfw \rangle$ be an element that is not in
  $R'$. Since the linear map
  $ \langle -,\sfw \rangle:V^* \to V^{\otimes m}$ has rank $n$, there
  is a basis $\{x,x_2',\ldots,x_n'\}$ for $V$ such that
  $\sfw=x\sfw_1'+x_2'\sfw_2'+\cdots +x_n'\sfw_n'$. It is now clear
  that $R'=\langle x^\perp,\sfw\rangle$.
\end{remark}

  We now know that the $D_x$'s have the same Hilbert series for
  all $x \in V-\{0\}$ but in order to complete the proof of \Cref{thm2} we
  need to construct a single sheaf of algebras $\cD$ on $\PP(V)$ whose fibers at the $\Bbbk$-points
  $x \in \PP(V)$ are the algebras $D_x$. The precise statement which we wish to prove is the following.
  
\begin{theorem}
	Let $A({\sfw})$ be a $3$-Calabi-Yau algebra generated by $x_1,\ldots,x_n$. Then there are flat sheaves $\cD$ and $\cA(\sfw)$ of graded $\cO_{\PP(V)}$-algebras and a morphism $\cD\to\cA(\sfw)$ such that for every $\Bbbk$-point $x \in \PP(V)$, the fibers of $\cD$ and $\cA(\sfw)$ at $x$ are the algebras $D_{x}$ (defined in \Cref{thm2}) and $A(\sfw)$, respectively, and the morphism $\cD\to\cA(\sfw)$ induces the projection
	\begin{equation*}
		D_{x}\to D_{x}/(\Omega_{x})=A(\sfw).
	\end{equation*}
\end{theorem}

\begin{proof}
We continue to work with a fixed basis $\{x_1,\ldots,x_n\}$ for $V$ and write $\sfw=x_{1}\sfw_{1}+\cdots+x_{n}\sfw_{n}$. 

Let $\{c_1,\ldots,c_n\}$ be the basis for $V^*$ that is dual to $\{x_1,\ldots,x_n\}$. 
Let $R=\Bbbk[c_{1},\ldots,c_{n}]$ be the polynomial ring on indeterminates $c_{1},\ldots,c_{n}$ placed in degree-one.
Thus, $\PP(V)=\Proj(R)$. Let $E$ be the $R$-algebra generated by $x_{1},\ldots,x_{n}$ modulo the relations
  \begin{equation*}
    c_{i}\sfw_{j}\,-\, c_{j}\sfw_{i} \; = \; [x_{i},\sfw_{j}] \; = \; 0,  \qquad 1\leq i,j\leq n,
  \end{equation*}
or, more succinctly, modulo the relations
  \begin{equation*}
    c_{i}\sfw_{j}\,-\, c_{j}\sfw_{i} \; = \; [V,\langle V^*,\sfw\rangle] \; = \; 0,  \qquad 1\leq i,j\leq n.
  \end{equation*}
  Viewing $E$ as a graded $R$-module  with respect to $c_{1},\ldots,c_{n}$, we define $\cD$ to be the quasi-coherent sheaf 
  of algebras on $\PP(V)$ whose sections over a basic open set $U_f=\Spec(R[f^{-1}]_0)$ is $E[f^{-1}]_0$. 
  Since $E$ has an algebra structure and also has a grading with $x_{1},\ldots,x_{n}$ placed in degree-one, 
  $\cD$ is a sheaf of graded   $\cO_{\PP(V)}$-algebras on $\PP(V)$.\footnote{The definition of $\cD$ does not depend on the basis $x_{1},\ldots,x_{n}$ for $V$.} 
  
  The fiber $\cD_{\fp}$ at a point $\fp\in\PP(V)$ (e.g., a $\Bbbk$-point, a closed point, or the generic point), which is by definition the pullback of $\cD$ along $\Spec(\Bbbk(\fp))  \to\PP(V)$, is obtained from $E$ by replacing the base ring $R$ by $\Bbbk(\fp)$ and $c_{i}$ in the relations by $\gamma_{i}:=c_{i}/c\in\Bbbk(\fp)$, where $c\in R_{1}$ is a fixed element not belonging to $\fp$.
 So the fiber $\cD_{\fp}$ is the $\Bbbk(\fp)$-algebra generated by $x_1,\ldots,x_n$ modulo the relations
   \begin{equation*}
    \c_{i}\sfw_{j}\,-\, \c_{j}\sfw_{i} \; = \; [V,\langle V^*,\sfw\rangle] \; = \; 0,  \qquad 1\leq i,j\leq n.
  \end{equation*}
  We will show that this algebra is $D_{x}$ with $\Bbbk$ replaced by $\Bbbk(\fp)$ and $x$ replaced by $y:=\c_{1}x_{1}+\cdots+\c_{n}x_{n}\in V\otimes_{\Bbbk}\Bbbk(\fp)$. Since $y^\perp={\rm span}_{\Bbbk(\fp)}\{ \c_i c_j-\c_j c_i \; | \; 1 \le i,j \le n\}\subseteq V^{*}\otimes_{\Bbbk}\Bbbk(\fp)$, $\langle y^{\perp},\sfw\rangle$
  is the span of the elements  
  $$
\langle \c_i c_j-\c_j c_i, \sfw \rangle \;=\; \langle \c_i c_j-\c_j c_i, x_{1}\sfw_{1}+\cdots+x_{n}\sfw_{n} \rangle \;=\; \c_{i}\sfw_{j}-\c_{j}\sfw_{i}.
$$
  Hence $\cD_{\fp}$ is the $\Bbbk(\fp)$-algebra generated by $x_1,\ldots,x_n$ modulo the relations
   \begin{equation*}
\langle y^{\perp},\sfw\rangle \; = \; [V,\langle V^*,\sfw\rangle] \; = \; 0;
  \end{equation*}
  i.e., $\cD_{\fp}$ is $D_y$ defined over $\Bbbk(\fp)$. In particular, if $\fp$ is a $\Bbbk$-point $x\in\PP(V)$, then the element $y\in V\otimes_{\Bbbk}\Bbbk(\fp)=V$ is equal to $x$ when viewed as a point of $\PP(V)$. So the fiber $\cD_{x}$ at $x$ is the algebra $D_{x}$.

%Consider the sheaf of $\cO_{\PP(V)}$-algebras
%$$
%\cA(\sfw) \; :=\; A(\sfw) \otimes_\Bbbk \cO_{\PP(V)}.
%$$
Similarly, we define the sheaf of graded $\cO_{\PP(V)}$-algebras $\cA(\sfw)$ to be the sheaf associated to the $R$-algebra $F$ generated by $x_{1},\ldots,x_{n}$ modulo the relations $\langle V^{*},\sfw\rangle$. As $\fp$ ranges over the points of $\PP(V)$, we obtain 
3-Calabi-Yau algebras $A(\sfw) \otimes_\Bbbk \Bbbk(\fp)$ over the fields $\Bbbk(\fp)$ all having the same Hilbert series as $A(\sfw)$. The
canonical surjective  algebra homomorphism $E \to F$ induces an epimorphism $\cD\to\cA(\sfw)$ that specializes to the projection $D_{y}\to A(\sfw)\otimes_\Bbbk\Bbbk(\fp)$.

%As $\fp$ ranges over the points of $\PP(V)$, with residue fields $\Bbbk(\fp)$, we obtain 
%3-Calabi-Yau algebras  $A(\sfw) \otimes_\Bbbk \Bbbk(\fp)$ over the fields $\Bbbk(\fp)$ all having the same %Hilbert series as $A(\sfw)$. 
%If $D$ is the algebra in \Cref{cor.cy}, $D \otimes_\Bbbk \Bbbk(\fp)$ is a 4-Calabi-Yau
%central extension of $A(\sfw) \otimes_\Bbbk \Bbbk(\fp)$ having the same Hilbert series as $D$. 

Applying \Cref{cor.cy} with $\Bbbk$ replaced by the residue field $\Bbbk(\fp)$ at a point $\fp$ in $\PP(V)$,
we conclude that the dimension of the  homogeneous component $\cD_n$ at the point $\fp$ is the same for all $\fp$. 
Hence by \cite[Lem.~II.8.9]{Hart}, each $\cD_n$ is locally free and $\cD$ is therefore flat. So is $\cA(\sfw)$ for the same reason.
\end{proof}

\begin{proof}[Alternative proof]
  Consider the free sheaf $\cV=V\otimes_{\Bbbk}\cO_{\PP(V)}$ over $\PP(V)$. We denote its dual by $\cV^*$; note that we have
  evaluation morphisms
  \begin{equation*}
    \mathrm{ev}:\cV^*\otimes \cV^{\otimes \ell}\to \cV^{\otimes (\ell-1)}. 
  \end{equation*}
  pairing $\cV^*$ against the leftmost tensorand.

  We will construct $\cD$ as $T\cV/(\cR_{m}\oplus \cR_{m+1})$, where $T\cV$ is the tensor
  algebra on the sheaf $\cV$ and
  $\cR_{m}$ and $\cR_{m+1}$ are subsheaves of $\cV^{\otimes m}$ and
  $\cV^{\otimes(m+1)}$, corresponding to \Cref{eq:rel1,eq:rel2}.

  Since over every open affine subset $\Spec(R)\subseteq \PP(V)$ the sheaf
  $\cV$ corresponds to the free $R$-module on the vector space $V$,
  $\sfw$ induces a section of $\cV^{\otimes(m+1)}$ (i.e., a morphism
  $\cO\to \cV^{\otimes(m+1)}$) denoted in the sequel by the same
  symbol.

  First, let $\cR_{m+1}\subseteq \cV^{\otimes(m+1)}$ be the image of
  the morphism
\begin{equation*}
  \begin{tikzpicture}[auto,baseline=(current  bounding  box.center)]
   \path[anchor=base] (0,0) node (vv*) {$\cV\otimes \cV^*$} +(3,0) node (vv*v) {$\cV\otimes \cV^*\otimes \cV^{\otimes(m+1)}$} +(7,0) node (vs) {$\cV\otimes \cV^{\otimes m}$} +(10,0) node (v) {$\cV^{\otimes(m+1)}$,};
   \draw[->] (vv*) to node[pos=.5,auto] {$\scriptstyle \mathrm{id}\otimes\sfw$} (vv*v);
   \draw[->] (vv*v) to node[pos=.5,auto] {$\scriptstyle \mathrm{id}_{\cV}\otimes\mathrm{ev}$} (vs);
   \draw[->] (vs) to node[pos=.5,auto] {$\scriptstyle [-,-]$} (v);      
  \end{tikzpicture}
\end{equation*}
where the last morphism is the commutator of the leftmost $\cV$ and
the tensor product $\cV^{\otimes m}$. $\cR_{m+1}$ specializes to
\Cref{eq:rel2} at $\Bbbk$-points of $\bP(V)$.

Similarly, we define $\cR_m\subseteq \cV^{\otimes m}$ specializing to
\Cref{eq:rel1} as follows.
Since the space of global sections of $\cO(1)$ is canonically isomorphic to $V^{*}$, we have a canonical epimorphism $V^{*}\otimes_{\Bbbk}\cO\to\cO(1)$. The dual of this morphism is a monomorphism $\cO(-1)\to V\otimes_{\Bbbk}\cO=\cV$, whose image specializes to the $1$-dimensional subspace $\Bbbk x\subseteq V$ at each $\Bbbk$-point $x\in\PP(V)$. Regard $\cO(-1)$ as a subsheaf of $\cV$ and consider its annihilator $\cO(-1)^\perp\subseteq
\cV^*$. We define $\cR_m$ as the image of the sheaf morphism
%\begin{equation*}
%  \begin{tikzpicture}[auto,baseline=(current  bounding  box.center)]
%   \path[anchor=base] (0,0) node (o) {$\cO(-1)^{\perp}$} +(3,0) node (ovs) {$\cO(-1)^{\perp}\otimes \cV^{\otimes(m+1)}$} +(6,0) node (vs) {$\cV^{\otimes m}$};
%   \draw[->] (o) to node[pos=.5,auto] {$\scriptstyle \mathrm{id}\otimes\sfw$} (ovs);
%   \draw[->] (ovs) to node[pos=.5,auto] {$\scriptstyle \mathrm{ev}$} (vs);   
%  \end{tikzpicture}
%\end{equation*}
\begin{equation*}
	\begin{tikzcd}
		\cO(-1)^{\perp}\ar[r,"\id\otimes\sfw"] & \cO(-1)^{\perp}\otimes\cV^{\otimes(m+1)}\ar[r,"\mathrm{ev}"] & \cV^{\otimes m}.
	\end{tikzcd}
\end{equation*}
As explained above, we can now define $\cD$ as
$T\cV/(\cR_m\oplus \cR_{m+1})$.

The sheaf $\cA(\sfw)$ is defined as $T\cV/(\cR)$, where $\cR$ is the image of
\begin{equation*}
	\begin{tikzcd}
		\cV^{*}\ar[r,"\id\otimes\sfw"] & \cV^{*}\otimes\cV^{\otimes(m+1)}\ar[r,"\mathrm{ev}"] & \cV^{\otimes m}
	\end{tikzcd}.
\end{equation*}

These basis-free definitions of $\cD$ and $\cA(\sfw)$ agree with those in the previous proof, and we can 
complete the proof using the same argument.
\end{proof}

\begin{proof}[Proof of \Cref{th.fl-nc}]
  This follows from \Cref{thm2}\cref{thm2.flatfamily} and the general remark that flatness of families of
  connected graded algebras transports over to their $\QGr$
  categories, recorded in \Cref{le.fl-aux} below.
\end{proof}

\begin{lemma}\label{le.fl-aux}
  Let $R$ be a commutative ring and $D$ a non-negatively graded
  $R$-algebra with $D_0=R$.
  If the homogeneous components $D_n$ are flat over $R$ then so is the
  abelian category $\QGr(D)$ where $\QGr(D)$ is the category  defined in \Cref{def.fam}.
\end{lemma}
\begin{proof}
  The $R$-linear category $\Gr(D)$ is equivalent to the category of modules over the small $R$-linear category 
  $\fd$ with objects $i\in \bZ$ and $\Hom_{\fd}(i,j)=D_{j-i}$. The hypothesis implies that $\fd$ is $R$-flat as a {\it linear} category so 
  $\Mod(\fd)$, and hence $\Gr(D)$, is flat over $R$ by \cite[Prop.~3.7]{lvdb-def}.  
Since the quotient functor $\Gr(D) \to \QGr(D)$ is exact and essentially surjective, flatness  of $\QGr(D)$ follows from that
of $\Gr(D)$ by \cite[Lem.~8.13]{lvdb-def}.
\end{proof}

%%%%%%%%%%%%%%%%%%%%%%%%%%%%%%%%%%%%%%%%%%%%%%%%%%%%%%%%%%%%%%%%%%%%%%%%%%%%%%%%%
\subsection{Automorphisms}\label{subse.nak}

Let $A=A(\sfw)$ and $D=D(\sfw,\sfp)$ be algebras satisfying the
assumptions of \Cref{th.3-stg}.  We now consider the relation between
$\Aut_{\sf gr}(A)$ and $\Aut_{\sf gr}(D)$.

\cref{pr.aut.lift} identifies a subgroup of $\Aut_{\sf gr}(A)$ that
consists of automorphisms that lift to $D$.

This result does not depend on \Cref{th.3-stg} and its proof will be
used in the proof of \Cref{th.3-stg}.

\begin{proposition}\label{pr.aut.lift}
  Let $A=A(\sfw)$ and $D=D(\sfw,\sfp)$ be algebras satisfying the
  hypotheses of \Cref{th.3-stg}. Let $\s\in\Aut_{\sf gr}(A)$.  If
  every $x_i$ is an eigenvector for $\s$, then $\s$ lifts to an
  automorphism of $D$.
\end{proposition}
\begin{proof}
By definition, the relations $f_i=\partial_i{\sfw}$ are
  the respective images of $x_i^*\otimes{\sfw}$ through the map
  \begin{equation}\label{eq:ev}
    \begin{tikzpicture}[auto,baseline=(current  bounding  box.center)]
      \path[anchor=base] (0,0) node (l)
      {$V^*\otimes V^{\otimes (m+1)}$} +(5,0) node (r)
      {$V^{\otimes m}$};
      \draw[->] (l) to node[pos=.5,auto]
      {$\scriptstyle
        \mathrm{ev}\otimes\mathrm{id}_{V^{\otimes m}}$} (r);
    \end{tikzpicture}
  \end{equation}
  that evaluates the two leftmost tensorands against each other, where
  $\{x_i^*\}\subseteq V^*$ is the basis dual to $\{x_i\}$.
  
By \cite[Thm.~1.1]{MS16}, the automorphisms of $A$ are
  (identified in the obvious manner with) precisely those linear
  automorphisms of $V=A_1$ which scale the superpotential
  ${\sf w}\in V^{\otimes{(m+1)}}$. In other words, the element
  ${\sf w}$ is an eigenvector for the action of
  $\Gamma=\langle \s\rangle$ on $V^{\otimes (m+1)}$ (with the tensor
  product being taken in the monoidal category of $\Gamma$-modules).

  Now suppose that  the generators $x_i$ are eigenvectors
  for $\s\in GL(V)$. Then each $\Bbbk f_i$ is the  image of
  the $\Gamma$-submodule $\Bbbk x_i^*\otimes \Bbbk{\sfw}$ of the left-hand side of the 
  $\Gamma$-module morphism \Cref{eq:ev}, so they are
  $\Gamma$-submodules of $V^{\otimes m}$.
Thus, each $f_i$ is an eigenvector for $\s^{\otimes m}$. Since $x_1$ is, by hypothesis, also an eigenvector for $\s$,
  it acts as a scalar on each of the
  defining relations \Cref{eq:d-rels} for $D$.
\end{proof}

In what follows we assume that the four conditions in \Cref{th.3-stg} are satisfied.

We have the following consequence of \Cref{pr.nak-ad} on how
the respective Nakayama automorphisms interact.

\begin{proposition}\label{pr.nak-3}
Let $A(\sfw)$ and $D(\sfw,\sfp)$ be Artin-Schelter regular algebras satisfying the hypotheses, and hence the conclusions, of \Cref{th.3-stg}.
If $Q={\rm diag}(q_{1},\ldots,q_{n})$, then the Nakayama automorphism of $D(\sfw,\sfp)$ scales $x_i$ by $p_i^{-1}q_i^{-1}$, i.e., $\nu_{\!{}_D}(x_i)=(p_iq_i)^{-1}x_i$ for all $i$. 
\end{proposition}
\begin{proof}
  This follows from \Cref{pr.nak-ad}, together with the fact that $\tau$, ``conjugation'' by $\Omega$, acts by
  \begin{equation*}
    x_i\mapsto p_i^{-1} x_i
  \end{equation*}
  and \cite[Thm.~1.8]{MS16} which, in our setting, says that
  $\nu_{\!{}_A}$ scales $x_i$ by $q_i^{-1}$.
\end{proof}

Applying this result to the algebras in \Cref{cor.alt-cy} we obtain the following result.

\begin{corollary}\label{cor.d-cy}
Suppose $q_1=1$. If $\sfp=(1,q_2^{-1}, \ldots, q_n^{-1})$, then $D({\sf w},\sfp)$ is 4-Calabi-Yau.
\end{corollary}
\begin{proof} 
  \Cref{pr.nak-3} implies that $\nu_{\!{}_D}=\id_D$ in this case.
\end{proof}

By \cite[Cor.~5.4]{RRZ17}, the homological determinant of the Nakayama automorphism is 1 for all noetherian Artin-Schelter regular algebras. 
By \cite[Thm.~1.6]{MS16}, the same result holds in the non-noetherian case provided the algebra is $m$-Koszul and Artin-Schelter regular.
It is conjectured at \cite[Conj.~4.12]{MS16} that this holds for {\it all} Artin-Schelter regular algebras. \Cref{pr.hdet-triv} below shows
that the normal extensions $D(\sfw,\sfp)$ produced by \cref{th.3-stg} have this property.

Before stating that result, we give a relatively self-contained proof
of the following auxiliary result. It is essentially \cite[Thm.~1.2]{MS16},
with the minor caveat that the homological determinant in
\cite{MS16} is inverse to that of \cite{jz}. Nevertheless,
\cite[Thm.~1.2]{MS16} appeals to \cite{wz-skw}, which agrees with
\cite{jz}. For this reason, \cite[Thm.~1.2]{MS16} is correct for
the version of the homological determinant used here and in
\cite{jz,wz-skw}, but not precisely as stated, for the version of
$\hdet$ defined in \cite[$\S$2.2]{MS16}.

Throughout the statement and the proof of \Cref{le.scale-w}, and in later sections, we identify each graded algebra automorphism of  
$A=TV/(R)$ (when the relations $R$ have degree $\ge 2$) with the  automorphism of $TV$ that agrees on the
degree-1 component $V$.

\begin{lemma}\label{le.scale-w}
  Let $A$ be an $m$-Koszul Artin-Schelter regular algebra and let $\sfw \in V^{\otimes \ell}$ be a $Q$-twisted superpotential such that
  $A \cong A(\sfw,\ell-m)$ (see \Cref{cor.alt-cy}). 
  If $\sigma \in \Aut_{\sf gr}(A)$, then $\sfw$ is an eigenvector for $\sigma^{\otimes\ell}$ with eigenvalue $\hdet(\sigma)$.
\end{lemma}
\begin{proof}
By \Cref{le.hdet-alt}, $\hdet(\sigma)^{-1}$ is  the scalar by which $\sigma\triangleright$ acts on the
  degree-$(-\ell)$ component of $\Ext^d_A(\Bbbk,\Bbbk)$. Resolving ${}_A\Bbbk$ by projective modules $P_\hdot$, this amounts to the
  action obtained by pre-composing morphisms in $\Hom_A(P_d,\Bbbk)$
  with the action of $(\varphi^{P_d})^{-1}$.

  Now, according to \cite[equation (4.2)]{dv2}, the minimal resolution
  $P_{\hdot}$ consists of free left $A$-modules generated by the
  partial derivatives of ${\sfw}$, and the left-most projective $P_d$ is
  precisely the free $A$-module with basis ${\sfw}$. It follows from
  that equation that $(\varphi^{P_d})^{-1}$ is the application of
  $(\sigma^{-1})^{\otimes (\ell+1)}$ to
  \begin{equation*}
    A\otimes \Bbbk {\sfw}\subseteq A^{\otimes(\ell+1)}.
  \end{equation*}
  \Cref{le.hdet-alt} now implies that $\hdet(\s)^{-1}$ is the eigenvalue associated to the eigenvector ${\sfw}$ of
  $(\sigma^{-1})^{\otimes\ell}$, hence the conclusion.
\end{proof}

\begin{theorem}\label{pr.hdet-triv}
Let $A=A(\sfw)$ and $D=D(\sfw,\sfp)$ be Artin-Schelter regular connected graded algebras satisfying all conditions in \Cref{th.3-stg}. Then $\hdet(\nu_{\!{}_D})=1$.
\end{theorem}
\begin{proof}
 \Cref{pr.nak-ad} shows that the restriction $\nu_{\!{}_D}|_A$ is
  $\tau|_A\circ \nu_{\!{}_A}$, where $\tau$ is defined on $D$ by
  \begin{equation*}
    \Omega x = \tau(x)\Omega
  \end{equation*}
  and $\tau|_A$ is its descent to $A$ modulo $\Omega$. 
 \Cref{th.hdets} shows that
  \begin{equation}\label{eq:hlh}
    \hdet(\nu_{\!{}_D}) = \lambda \hdet(\nu_{\!{}_D}|_{A}),
  \end{equation}
  where $\lambda$ is defined by $\nu_{\!{}_D}(\Omega)=\lambda \Omega$. Hence, by \cref{pr.nak-ad} and the
  multiplicativity of $\hdet$ (\cite[Prop.~2.5]{jz}),   
  \begin{equation}\label{eq:hlhh}
    \hdet(\nu_{\!{}_D}) = \lambda\cdot \hdet(\tau|_A)\cdot \hdet(\nu_{\!{}_A}).
  \end{equation}
But $\hdet(\nu_{\!{}_A})=1$ by \cite[Thm.~1.6]{MS16}, so the
  rightmost factor in \Cref{eq:hlhh} is 1.

  On the other hand, since $\tau$ fixes $\Omega$, the equation $\nu_{\!{}_D}|_A=\tau\circ\nu_{\!{}_A}$
  shows that $\lambda$ is defined by
  \begin{equation}\label{eq:olo}
   \nu_{\!{}_A}(\Omega)=\lambda\Omega, 
  \end{equation}
  where we have abused notation by using $\nu_A$ to denote the automorphism of the free algebra $TV$ that agrees with $\nu_A$ on $A_1=V$
  (see the discussion preceding   \Cref{le.scale-w}).

  By \cite[Thm.~1.6]{MS16} and \Cref{le.scale-w}, $\nu_{\!{}_A}$
  fixes ${\sf w}$, and \cite[Thm.~1.8]{MS16} shows that the same
  automorphism scales $x_1$ by $q_1^{-1}$. It now follows from the
  fact that $x_1\Omega$ is a sum of monomials appearing in ${\sf w}$ that
  $\nu_{\!{}_A}$ scales $\Omega$ by $q_1$. In other words, \Cref{eq:olo}
  implies $\lambda=q_1$.

  Finally,  \cite[Thm.~1.2]{MS16} shows that
  $\hdet(\tau|_A)$ is equal to the eigenvalue of the $\tau|_A$-eigenvector
  ${\sf w}$ which condition \cref{item.param} in \Cref{th.3-stg} shows
  is equal to $q_1^{-1}$.

  To summarize:
  \begin{itemize}
  \item the leftmost factor in the right-hand side of \Cref{eq:hlhh} is $q_1$;
  \item the middle factor is $q_1^{-1}$;
  \item the rightmost factor is $1$.   
  \end{itemize}
  The desired conclusion that $\hdet(\nu_{\!{}_D})=1$  now follows
  from \Cref{eq:hlhh}.
\end{proof}

\subsection{Zhang twists}

We now consider how our construction of normal extensions behaves with respect to Zhang twists. 
Zhang twisting produces new graded algebras by using families of graded $k$-linear automorphisms called 
twisting systems \cite{Z-twist}. 
The idea extends the idea in \cite[\S8]{ATV2}. 
Here we only consider the Zhang twist associated to a graded algebra automorphism of $A$.

\begin{definition}
	Let $A$ be a graded algebra and $\sigma$ an automorphism on $A$. The \textsf{Zhang twist} ${}^{\sigma}A$ of $A$ by $\sigma$ is defined by ${}^{\sigma}A:=A$ as an graded vector space with a new multiplication $a*b:=\sigma^{\deg b}(a)b$ where $a,b\in A$ are homogeneous elements.\footnote{Our definition of Zhang twist agrees with that of \cite{RRZ14} but differs from \cite{MS16}.}
\end{definition}

For each $f\in V^{\otimes r}$, set ${}^{\sigma}f:=(\id_{V}\otimes\sigma\otimes\cdots\otimes\sigma^{r-1})(f)$. Then the relations in ${}^{\sigma}A$ are given by ${}^{\sigma}f$ where $f$ runs over all relations in $A$.

\begin{proposition}
	Let $A({\sfw})$ and $D(\sfw,\sfp)$ be Artin-Schelter regular algebras as in \cref{th.3-stg}. Let $\sigma$ be an automorphism of $A({\sfw})$ such that $\sigma(x_{i})=s_{i}x_{i}$ for some $s_{i}\in\Bbbk^{\times}$. Then $\sigma$ lifts to an automorphism of $D(\sfw,\sfp)$, and
	\begin{equation*}
		{}^{\sigma}D(\sfw,\sfp)\cong D({}^{\sigma}\sfw,\sfp')
	\end{equation*}
	where $\sfp'=(p'_{1},\ldots,p'_{n})$ is defined by $p'_{i}=p_{i}s_{1}s_{i}^{m}(\hdet_{A}(\sigma))^{-1}$.
\end{proposition}
\begin{proof}
	As observed in \cref{pr.aut.lift} and its proof, $\sigma$ lifts to $D$, and $\sigma^{\otimes m}$ scales each $f_{i}$ in $V^{\otimes m}$. 
	The algebra ${}^{\sigma}D(\sfw,\sfp)$ is defined by the relations ${}^{\sigma}f_{i}={}^{\sigma}[x_{i},\partial_{1}\sfw]_{p_{i}}=0$ ($i=2,\ldots,n$). By \cite[Thm.~5.5]{MS16}, ${}^{\sigma}A(\sfw)=A({}^{\sigma}\sfw)$. Note that $\partial_{i}{}^{\sigma}\sfw=s_{i}^{-1}(\hdet_{A}(\sigma))({}^{\sigma}f_{i})$. By \cref{le.scale-w},
	\begin{equation*}
		(\hdet_{A}(\sigma))\sfw=\sigma^{\otimes(m+1)}(\sfw)=\sigma^{\otimes(m+1)}(x_{1}f_{1}+\cdots+x_{n}f_{n})=s_{1}x_{1}\sigma^{\otimes m}(f_{1})+\cdots+s_{n}x_{n}\sigma^{\otimes m}(f_{n})
	\end{equation*}
	and hence $\sigma^{\otimes m}(f_{i})=s_{i}^{-1}(\hdet_{A}(\sigma))f_{i}$. Therefore
	\begin{equation*}
		{}^{\sigma}[x_{i},\partial_{1}\sfw]_{p_{i}}=x_{i}{}^{\sigma}(\sigma^{\otimes m}(f_{1}))-p_{i}{}^{\sigma}f_{1}\sigma^{m}(x_{i})=s_{i}^{-1}(\hdet_{A}(\sigma))x_{i}{}^{\sigma}(f_{1})-p_{i}{}^{\sigma}f_{1}\sigma^{m}(x_{i})=[x_{i},\partial_{1}{}^{\sigma}\sfw]_{p'_{i}}.
	\end{equation*}
	The defining ideal of ${}^{\sigma}D(\sfw,\sfp)$ is generated by $\partial_{i}{}^{\sigma}\sfw$ and $[x_{i},\partial_{1}{}^{\sigma}\sfw]_{p'_{i}}$ ($i=2,\ldots,n$). 
\end{proof}

%%%%%%%%%%%%%%%%%%%%%%%%%%%%%%%%%%%%%%%%%%%%%%%%%%%%%%%%%%%%%%%%%%%%%%%%%%%%%%%%%
%%%%%%%%%%%%%%%%%%%%%%%%%%%%%%%%%%%%%%%%%%%%%%%%%%%%%%%%%%%%%%%%%%%%%%%%%%%%%%%%%
\section{Proof of the main theorem}
\label{sec.proof}

In this section we prove \Cref{th.3-stg}. Let $A=A(\sfw)$ and $D=D(\sfw,\sfp)$ be the algebras as in the theorem.

\subsection{}

In this subsection we show that \cref{item.reg} $\Rightarrow$ \cref{item.eigenvec} $\Leftrightarrow$ \cref{item.param} $\Rightarrow$  \cref{eq:sk}.

The proof that \cref{item.reg} $\Rightarrow$ \cref{item.eigenvec} also shows that $\Omega$ is a regular element in $D(\sfw,\sfp)$. 

\cref{item.reg} $\Rightarrow$ \cref{item.eigenvec}
We first show that the relations \cref{eq:d-rels} form a minimal set of relations 
for $D(\sfw,\sfp)$. Suppose not. Since $\pd_{1}\sfw,\ldots,\pd_{n}\sfw$ are linearly independent (see \cref{subsec.not.cycder}), 
there are scalars $a_{i},b_{ij},c_{ij}\in\Bbbk$, not all zero, such that
\begin{equation*}
	\sum_{i=2}^{n}a_{i}[x_{i},\pd_{1}\sfw]_{p_{i}}\, + \, 
	\sum_{i=2}^{n}\sum_{j=1}^{n}\big(b_{ij}x_{j}(\pd_{i}\sfw)\,+\,c_{ij}(\pd_{i}\sfw)x_{j}\big)\;=\; 0
\end{equation*}
in $TV$. In other words, 
\begin{equation}\label{eq.lincomb}
	\sum_{i=2}^{n}a_{i}x_{i}(\pd_{1}\sfw)\,+\, \sum_{i=2}^{n}\sum_{j=1}^{n}b_{ij}x_{j}(\pd_{i}\sfw)
	\;=\; 
	\sum_{i=2}^{n}a_{i}p_{i}(\pd_{1}\sfw)x_{i}\,-\,\sum_{i=2}^{n}\sum_{j=1}^{n}c_{ij}(\pd_{i}\sfw)x_{j}.
\end{equation}
Let $R$ be the subspace of $TV$ spanned by $\pd_{1}\sfw,\ldots,\pd_{n}\sfw$. The 
left-hand side of \cref{eq.lincomb} belongs to $V \otimes R$ and the right-hand side of \cref{eq.lincomb} belongs to $R \otimes V$.
Since $(V\otimes R)\cap(R\otimes V)=\Bbbk\sfw$  by \cite[Prop.~2.12]{MS16}, the left-hand side of \cref{eq.lincomb}
is a scalar multiple of $\sfw=x_{1}(\pd_{1}\sfw)+\cdots+x_{n}(\pd_{n}\sfw)$. This implies that all $a_i=0$ and all $b_{ij}=0$.
Similarly, the right-hand side of \cref{eq.lincomb} is a scalar multiple of $\sfw=q_{1}(\pd_{1}\sfw)x_{1}+\cdots+q_{n}(\pd_{n}\sfw)x_{n}$ which
implies that all $a_i=0$ and all $c_{ij}=0$. This is a contradiction. 
Hence \cref{eq:d-rels} is a minimal set of relations.

Thus the minimal free resolution of ${}_{D}\Bbbk$ is of the form
\begin{equation*}
	0\to D(-2m-1)\to D(-2m)^{n}\to D(-m)^{n-1}\oplus D(-m-1)^{n-1}\to D(-1)^{n}\to D\to {}_{D}\Bbbk\to 0.
\end{equation*}
That of ${}_{A}\Bbbk$ is of the form
\begin{equation*}
	0\to A(-m-1)\to A(-m)^{n}\to A(-1)^{n}\to A\to{}_{A}\Bbbk\to 0.
\end{equation*}
These resolutions determine the Hilbert series of $D$ and $A$ and it is then easy to 
see that $h_{D}(t)=h_{A}(t)(1-t^{m})^{-1}$. Therefore by \cref{pr.enh}, $\Omega$ is a regular element of $D(\sfw,\sfp)$.
  
  Since $\Omega$ is regular, we can define an automorphism $\tau \in \Aut_{\sf gr}(D(\sfw,\sfp))$ by the formula in  \cref{eq:tau}. 
The argument in the proof of \cref{le.stg-cntr} shows that  $\tau(x_{i})=p_{i}^{-1}x_{i}$ for each $i=1,\ldots,n$. 
 Of course, $\tau$ can also be viewed as an element in $\GL(V)$ and, as observed in the proof of \cref{pr.aut.lift}, 
  $\sfw\in V^{\otimes(m+1)}$ and   $f_{1}\in V^{\otimes m}$ are eigenvectors for $\tau^{\otimes(m+1)}$
  and $\tau^{\otimes m}$, respectively. 
  Since $\tau(\Omega)=\Omega$, $\tau^{\otimes m}(f_1)=f_1$.  
  Since $\sfw=x_{1}f_{1}+\cdots+x_{n}f_{n}$, the  $\tau^{\otimes(m+1)}$-eigenvalue of $\sfw$ is the same as the $\tau$-eigenvalue of $x_{1}$. 
  Therefore   $\tau^{\otimes(m+1)}(\sfw)=p_{1}^{-1}\sfw =q_{1}^{-1}\sfw$. Since the automorphism $\varphi$ defined in \cref{th.3-stg}\cref{item.eigenvec} is the
  inverse of $\tau$,  \cref{item.eigenvec} holds.
 
\cref{item.param} $\Leftrightarrow$ \cref{item.eigenvec} 
  Since ${\sfw} = \sum_{i,j}x_i \sfM_{ij}x_j$, $\varphi^{\otimes(m+1)}$ scales $\sfw$ by $q_{1}$ if and only if it scales $x_i x_{l_{1}}\cdots x_{l_{m-1}} x_j$ by $q_{1}$ for all $i,j$ and all words $x_{l_{1}}\cdots x_{l_{m-1}}$ appearing in $\sfM_{ij}$. But $\varphi^{\otimes(m+1)}(x_i x_{l_{1}}\cdots x_{l_{m-1}} x_j)=p_{i}p_{j}p_{l_{1}}\cdots p_{l_{m-1}}x_i x_{l_{1}}\cdots x_{l_{m-1}} x_j$ by definition of $\varphi$.

  \cref{item.param} $\Rightarrow$ \cref{eq:sk}
  Fix $1\leq i,j\leq d$ and write
  $\sfM_{ij}=\sum_{l}a_{l}x_{l_{1}}\ldots x_{l_{m-1}}$ where every
  $a_{l}\neq 0$ and $l=(l_{1},\ldots,l_{m-1})$. As described in the
  proof of \cref{le.stg-cntr}, $[x_{j},\Omega]_{p_{j}}=0$ for each
  $j$. Hence
  \begin{equation}\label{EqCondInThm}
  [\Omega,\sfM_{ij}]_{q_{1}^{-1}p_{i}p_{j}}
  \;=\; \sum_{l}a_{l}[\Omega,x_{l_{1}}\ldots x_{l_{m-1}}]_{q_{1}^{-1}p_{i}p_{j}}
  \;=\; \Omega\sum_{l}a_{l}(1-q_{1}^{-1}p_{i}p_{j}p_{l_{1}}\ldots p_{l_{m-1}})x_{l_{1}}\ldots x_{l_{m-1}}.
  \end{equation}
  The right-most side is zero by \cref{item.param}.

%%%%%%%%%%%%%%%%%%%%%%%%%%%%%%%%%%%%%%%%%%%%%%%%%%%%%%%%%%%%%%%%%%%%%%%%%%%%%%%%%
\subsection{}
To complete the proof of \Cref{th.3-stg} we show that \cref{eq:sk} $\Rightarrow$  \cref{item.reg}.  The proof of this occupies the rest of \S\ref{sec.proof}.

First we identify a candidate for the minimal projective resolution of the trivial module ${}_D\Bbbk$. 
The minimal resolution of $\Bbbk$ as a left $A$-module is
\begin{equation}\label{eq:res-a}
	\begin{tikzcd}
          0\arrow[r] & A(-m-1) \ar[r, "\cdot {\sfx}^t"]&
          A(-m)^{n}\arrow[r,"\cdot \sfM"] & A(-1)^{n}\arrow[r,"\cdot {\bf
            x}"] & A\arrow[r] & \Bbbk\arrow[r] & 0.
	\end{tikzcd}
\end{equation}

By definition, the relations defining  $D$ are $f_i$ and
$g_i:=[x_i,\partial_1{\sfw}]_{p_i}$ for $2\le i\le n$.  We aggregate
them into two column vectors of size $2(n-1)\times 1$, namely
$$
  {\sfg}_l:=(q_1 p_2^{-1}g_2, \ldots,  q_1 p_{n}^{-1}g_{n}  \; | \;     q_2f_2, \ldots ,q_{n}f_{n})^t 
   \qquad \text{and} \qquad 
   {\sfg}_r:=(f_2 , \ldots , f_n  \; | \;  g_2, \ldots,  g_n)^t.
$$
The subscripts $l$ and $r$ denote ``left'' and ``right'' respectively. 
There are unique matrices $\sfM_l$ and $\sfM_r$ such that
\begin{equation}\label{eq:lr}
  {\sfx}^t \sfM_l \;=\; {\sfg}_l^t 
  \qquad \text{and} \qquad \sfM_r{\sfx} \;=\; {\sfg}_r. 
\end{equation}
Let $P_\hdot$ denote the sequence
\begin{equation}\label{eq:min-d}
  \begin{tikzcd}[auto,baseline=(current  bounding  box.center)]
    D(-2m-1) \ar[r, "\cdot {\sfx}^t"]& D(-2m)^{n}\arrow[r,"\cdot
    \sfM_l"]& D(-m)^{n-1}\oplus D(-m-1)^{n-1}\arrow[r,"\cdot \sfM_r"] &
    D(-1)^{n}\arrow[r,"\cdot {\sfx}"] & D
  \end{tikzcd}
\end{equation}
 of left $D$-modules and $D$-module homomorphisms.
Eventually, we will show that
  \begin{equation}\label{eq:main-seq}
    0\to P_{\hdot}\to \Bbbk\to 0
  \end{equation}
is a minimal resolution of the trivial $D$-module ${}_D\Bbbk$.

  As usual, when working with resolutions, we place $\Bbbk$ in homological degree $-1$ in \Cref{eq:res-a}, \Cref{eq:min-d}, and
   \Cref{eq:main-seq}.

We denote the homology of the complex
  \Cref{eq:main-seq} by $H_i$, $0\le i\le 4$. Each $H_i$ is a graded vector space; its homogeneous components will be denoted by
  $H_{i\ell}$. The dimension of   $H_{i\ell}$ will be denoted by $h_{i\ell}$ or $(h_i)_\ell$.

\subsection{Notation}\label{not.ind}
  Let $X$ be a matrix. For subsets $S$ and $T$ of the set of rows and
  vectors respectively, we denote by $X_{S,T}$ (or $X_{ST}$ if there
  is no danger of ambiguity) the matrix consisting of only those
  entries of $X$ whose row and column indices belong to $S$ and
  $T$. Moreover, if the sets $S$ and $T$ are one-sided intervals, we
  simply substitute the corresponding inequality for the set.

  For example, $X_{\ge 2,\le 6}$ is the sub-matrix of $X$
  consisting of entries $x_{ij}$ for $i\ge 2$ and $j\le 6$.

  We apply the same convention to row or column vectors, except that
  we adorn these with a single symbol as in ${\sfv}_{\ge 5}$ denoting
  the vector consisting of the entries of ${\sfv}$ with index $\ge 5$. 
  
  We use the symbol $\hdot$ for no condition at all,
  e.g., $X_{\hdot\hdot}=X$, ${\sfv}_{\hdot}={\sfv}$, and
  $X_{\ge 8,\hdot}$ is the matrix obtained from $X$ by deleting
  the first 7 rows.

  The hat indicates the negation of the condition in question (or the
  complement of the set). For instance, $X_{\hdot \widehat{j}}$ is
  the matrix obtained from $X$ by discarding its $j^{th}$ column.

We define 
  \begin{align*}
  J_h  & \;  :=\; \text{the matrix obtained by adding a row of $0$'s on top of $q_1 \mathrm{diag}(p_2^{-1},\ldots, p_{n}^{-1})$}\quad and
  \\
J_v  & \;  :=\; \text{the matrix obtained by adding a column of $0$'s to the left of $\mathrm{diag}(p_2\ldots p_n)$}.
\end{align*}

\subsection{}

The next lemma follows immediately from the definitions of $\sfM_l$ and $\sfM_r$ in \Cref{eq:lr}.
The horizontal and vertical bars in its statement split the respective matrices into blocks of equal size.

\begin{lemma}\label{le.mlr}
  The $2(n-1)\times n$ matrix $\sfM_r$ is
  \begin{equation*}
    \begin{pmatrix}[c]
      \sfM_{\widehat{1}\hdot}\\ \hline {\sfx}_{\ge 2} \sfM_{1\hdot}-f_1 J_v
    \end{pmatrix}
  \end{equation*}
  The $n\times 2(n-1)$ matrix $\sfM_l$ is
  \begin{equation*}
    \begin{pmatrix}[c|c]
      -\sfM_{\hdot 1}({\sfx}_{\ge 2})^t + f_1 J_h& \sfM_{\hdot \widehat{1}}
    \end{pmatrix}.
    \hfill   
  \end{equation*}
\end{lemma}

\begin{lemma}\label{le.cplx}
If condition \cref{eq:sk} in \Cref{th.3-stg} holds, then the diagram  $0\to P_{\hdot}\to \Bbbk\to 0$ in  \Cref{eq:min-d} is a complex.
\end{lemma}
\begin{proof}
  The composition $P_1\to P_0\to \Bbbk$ is zero because, the image of the map $P_1 \to P_0=D$ is the maximal ideal $(x_1,\ldots,x_n)$.
The definitions of $\sfM_l$ and $\sfM_r$, respectively, imply that the compositions $P_4\to P_3 \to P_2$ and $P_2\to P_1 \to P_0$ are zero.
It remains to show that the composition
\begin{equation*}
  \begin{tikzcd}[auto,baseline=(current  bounding  box.center)]
    D(-2m)^{n}\arrow[r,"\cdot \sfM_l"]& D(-m)^{n-1}\oplus
    D(-m-1)^{n-1}\arrow[r,"\cdot \sfM_r"] & D(-1)^{n}
  \end{tikzcd}
\end{equation*}
is zero, i.e., that the entries of the $n\times n$ matrix
$\sfM_l\sfM_r\in M_n\big(V^{\otimes(2m-1)}\big)$ are zero in $D$.

To this end, we consider the product of the $i^{th}$ row of $\sfM_l$ by the $j^{th}$
column of $\sfM_r$.  The $i^{th}$ row of $\sfM_l$ is
\begin{equation}\label{eq:rw}
  \begin{pmatrix}[c|c]
    -\sfM_{i1}({\sfx}_{\ge 2})^t+ f_1(J_h)_{i\hdot} & \sfM_{i \widehat{1}}
  \end{pmatrix},
\end{equation}
and the $j^{th}$ column of $\sfM_r$ is
\begin{equation}\label{eq:clmn}
  \begin{pmatrix}
    \sfM_{\widehat{1}j}\\ \hline {\sfx}_{\ge 2}\sfM_{1j}-f_1(J_v)_{\hdot j}
  \end{pmatrix}.
\end{equation}
Since ${\sfx}^t \sfM=(Q{\sff})^t$, the product of the left-hand side of \Cref{eq:rw} and the top half of \Cref{eq:clmn} is
\begin{equation}\label{eq:1/2}
  -\sfM_{i1}(q_j f_j-x_1 \sfM_{1j}) + (1-\delta_{i1})q_1 p_i^{-1} f_1 \sfM_{ij}
\end{equation}
Since  $\sfM{\sfx}={\sff}$, the product of the right-hand half of \Cref{eq:rw} by the bottom half of \Cref{eq:clmn} is
\begin{equation}
  \label{eq:and1/2}
  (f_i-\sfM_{i1}x_1)\sfM_{1j}-(1-\delta_{1j})p_j \sfM_{ij}f_1.
\end{equation}
Upon adding \Cref{eq:1/2} and \Cref{eq:and1/2} the terms containing
$x_1$ cancel out, and the remaining sums, in the four cases
\begin{enumerate}
  \renewcommand{\labelenumi}{(\arabic{enumi})}
\item $i=1=j$;
\item $i=1\ne j$;
\item $i\ne 1=j$;
\item $i\ne 1\ne j$  
\end{enumerate}
precisely coincide with the respective expressions in the statement of
\Cref{th.3-stg}. The conclusion then follows from our assumption that
these vanish in $D$.
\end{proof}

\begin{remark}\label{re.blk}
  A simultaneous change of basis for the spaces spanned by $x_i$ and
  $f_i$ for $i\ge 2$ induces, as explained in \cite[$\S$2]{AS87}, a
  conjugation of the diagonal matrix $Q$ by a block diagonal scalar
  matrix of the form $T=\diag(1,T')$, where $T'$ is an $(n-1)\times (n-1)$ matrix. The resulting matrix $TQT^{-1}$ is then of the form
  $\diag(q_1,Q')$ for some $Q'\in M_{n-1}(\Bbbk)$.

  Under this basis change the above proof of \Cref{le.cplx} consists
  of an identity involving the entries of the lower right-hand corner
  $Q'$ of $Q$, and hence functions equally well for any diagonalizable
  $Q'$. But then, by Zariski continuity, it is valid for any $Q'$
  whatsoever. This justifies the remark in \S\ref{re.gen} that the proof of \Cref{th.3-stg} extends to the case $Q=\diag(1,Q')$.
\end{remark}

\subsection{Inductive argument}

Define $Z:=\{a \in D \;| \; a\Omega=0\}$ and $z_i:=\dim_\Bbbk(Z_i)$.

\begin{lemma}\label{le.zqq}
Suppose condition \cref{eq:sk} in \Cref{th.3-stg} holds. If
$(h_3)_{k-m} =  z_{k-2m-1} =  z_{k-2m}  =  0$,
then $(h_3)_{k}=0$.
\end{lemma}
\begin{proof}
Let $  {\bf \varphi}^t = (\varphi_1,\ldots,\varphi_n)$   
be an element in $D_{k-2m}^d$ representing a class in $(H_3)_k$. Thus,  ${\bf \varphi}^t \sfM_l=0$ in $D$.

Let $\eta:=\varphi^{t}\sfM$. Because $f_1=\Omega=0$ in $A$, the expression for $\sfM_l$ in \Cref{le.mlr} implies that
\begin{equation*}
	\eta\begin{pmatrix}-(\sfx_{\geq 2})^{t} & \\ & 1_{n-1}\end{pmatrix}=0
\end{equation*}
in $A$. Hence $\eta_{1}(\sfx_{\geq 2})^{t}=0$ and $\eta_{\geq 2}=0$ in $A$. It also follows that
\begin{equation*}
	\eta_{1}x_{1} \;=\; \eta_{1}x_{1}+\eta_{\geq 2}\sfx_{\geq 2}  \;=\;  \varphi^{t}\sfM\sfx  \;=\;  0
\end{equation*}
in $A$. We have obtained $\eta_{1}\sfx^{t}=0$. The exact sequence \Cref{eq:res-a} ensures $\eta_{1}=0$, and hence $\varphi^{t}\sfM=0$ in $A$.

Again by the exactness of \Cref{eq:res-a}, ${\bf \varphi}^t$ is in the image of $\cdot{\sfx}^t$ modulo
$\Omega$, i.e.,
\begin{equation}\label{eq:phi}
  {\bf \varphi}^t \; = \; u{\sfx}^t + {\sfy}^t \Omega
\end{equation}
for some $u\in D_{k-2m-1}$ and a column vector ${\sfy}$ with entries
in $D_{k-3m}$. The equations \Cref{eq:lr} show that the first summand
of \Cref{eq:phi} is annihilated by right multiplication by $\sfM_l$, and
hence ${\sfy}^t\Omega \sfM_l=0$ in $D$.

Since $[\Omega, \sfM_{ij}]_{q_1^{-1}p_ip_j}=0$,  \Cref{le.mlr} implies that
$\Omega \sfM_l$ equals $\sfM_l\Omega$ up to multiplication on the left by an
 invertible $n\times n$ diagonal matrix and on the right by an
invertible diagonal matrix of size $2(n-1)$, namely
\begin{equation*}
  \Omega \sfM_{l} \; =\; \diag(q_{1},p_{2},\ldots,p_n)\, \sfM_{l}\Omega\,\diag(p_{2}^{-1},\ldots,p_n^{-1},q_{1}^{-1}p_{2},\ldots,q_{1}^{-1}p_n).
\end{equation*}
It follows that ${\sfy}^t\diag(q_{1},p_{2},\ldots,p_n)\sfM_l$ belongs
to $Z_{k-2m}^{d-1}\oplus Z_{k-2m-1}^{d-1}$ which is assumed to be
zero, and hence ${\sfy}^t\diag(q_{1},p_{2},\ldots,p_n)$ represents
a class in $(H_3)_{k-m}$ which is again assumed zero. In conclusion we
have ${\sfy}^t\diag(q_{1},p_{2},\ldots,p_n)=v{\sfx}^t$ for some
$v\in D(-2m-1)$ and hence, by \Cref{eq:phi},
\begin{equation*}
  {\bf \varphi}^t \;=\;  u{\sfx}^t + v{\sfx}^t\diag(q_{1}^{-1},p_{2}^{-1},\ldots,p_n^{-1})\Omega   \;=\; (u+ v\Omega){\sfx}^t.
\end{equation*}
In other words $\varphi^t$ is a boundary, so the homology $(H_3)_k$ vanishes, as desired.
\end{proof}

Now define
\begin{equation*}
 d_k \; :=\; \dim_\Bbbk(D_k), \qquad d'_k \; := \; \text{the coefficient of }t^k\text{ in } h_A(t) \left(1-t^m\right)^{-1}, 
  \qquad \text{and} \qquad e_k: = d'_k-d_k. 
\end{equation*}
It follows from the exact sequence $0 \to Z(-m) \to D(-m) \stackrel{\cdot \Omega}{\longrightarrow} D \to A \to 0$
that $e_k\ge 0$ (as noted in \Cref{eq:diff}) and 
\begin{equation}
  \label{eq:ez}
  e_k \;= \;  e_{k-m}+z_{k-m}. 
\end{equation}
for all $k$. 
Moreover, subtracting
\begin{equation*}
  -\delta_{k,0}+d_k-n d_{k-1}+(n-1)d_{k-m-1}+(n-1)d_{k-m}-nd_{k-2m}+d_{k-2m-1} \;=\;  (h_2)_k-(h_3)_k+(h_4)_k
\end{equation*}
from the same expression with $d'$  in place of $d$ (in which case the  right-hand side $=0$) we get
\begin{equation}\label{eq:e}
  e_k-n e_{k-1}+(n-1)e_{k-m-1}+(n-1)e_{k-m}-n e_{k-2m}+e_{k-2m-1} \;=\; -(h_2)_k+(h_3)_k-(h_4)_k.
\end{equation}

This is now sufficient preparation for

\begin{lemma}\label{le.ez0}
  Suppose condition \cref{eq:sk} in \Cref{th.3-stg} holds. Let $k$ be an integer. If  $e_i=0$ for all  $i\le k-1$ and $z_j=0$ for all  $j\le k-2m$, then $e_k=z_{k+1-2m}=0$.
\end{lemma}
\begin{proof}
 By \Cref{le.zqq}, the vanishing of the $z_j$'s implies that $(h_3)_k=0$ 
  (by an induction argument since $(h_3)_\ell=0$ when  $\ell<0$). Equation \Cref{eq:e} then gives
  \begin{equation*}
    0 \; \le \;  e_k \,=\, -(h_2)_k-(h_4)_k \; \le \; 0
  \end{equation*}
  so $e_k =0$ and $ (h_2)_k+(h_4)_k =0$. 

Finally, $z_{k+1-2m}=0$ because the $e_i$'s vanish and $ e_{k+1-m} = e_{k+1-2m}+z_{k+1-2m}$ (by \Cref{eq:ez}).
\end{proof}

To complete the proof of \cref{th.3-stg} we must show that $D$ is Artin-Schelter regular of dimension 4.
By \Cref{pr.enh}, $D$ is Artin-Schelter regular of dimension 4 if $h_D(t) =h_A(t)\left(1-t^m\right)^{-1}$, i.e., if $e_k=0$ for all $k$. 
Certainly, $e_k=z_k=0$ if $k<0$;  \Cref{le.ez0} then provides the induction step thereby showing that the vanishing 
of $e_k$ and $z_k$ perpetuates for all $k \ge 0$. 

The proof of \cref{th.3-stg} is now complete.

%%%%%%%%%%%%%%%%%%%%%%%%%%%%%%%%%%%%%%%%%%%%%%%%%%%%%%%%%%%%%%%%%%%%%%%%%%%%%%%%%
%%%%%%%%%%%%%%%%%%%%%%%%%%%%%%%%%%%%%%%%%%%%%%%%%%%%%%%%%%%%%%%%%%%%%%%%%%%%%%%%%
\section{Examples}
\label{sect.examples}

In this section we apply \Cref{th.3-stg} to some important classes of 3-dimensional Artin-Schelter regular algebras thereby producing some
new 4-dimensional Artin-Schelter regular algebras.

%%%%%%%%%%%%%%%%%%%%%%%%%%%%%%%%%%%%%%%%%%%%%%%%%%%%%%%%%%%%%%%%%%%%%%%%%%%%%%%%%
\subsection{The Calabi-Yau case}

As observed in \Cref{cor.cy}, if $A(\sfw)$ is 3-Calabi-Yau, then $\sfp=(1,\ldots,1)$ is good so the procedure in \cref{th.3-stg}
produces at least one 4-Calabi-Yau algebra that is a normal extension of $A(\sfw)$. It is, in fact, a  central extension. 
The 3-Calabi-Yau algebras of finite Gelfand-Kirillov dimension have been classified by 
Mori-Smith \cite{ms-cy} and Mori-Ueyama \cite{mu-cy}.
They are noetherian domains.

\begin{proposition}\label{pr.cy-gk}
For all the 3-Calabi-Yau algebras $A$ classified in \cite[Table 2]{ms-cy}
and \cite[Table 2]{mu-cy} and all numberings of the generators, the procedure 
 described in \Cref{cor.cy} produces a noetherian $4$-dimensional Artin-Schelter regular 
  central extension of $A$.  
\end{proposition}

Non-noetherian Calabi-Yau algebras are less well understood, but they do exist
and a number of them fit into our framework. 
For example,   \cite{e}
constructs, usually non-noetherian, $3$-Calabi-Yau algebras from
superpotentials ${\sfw}$ defined in terms of combinatorial objects
known as Steiner systems. The idea for using Steiner (triple)
systems to construct 3-Calabi-Yau algebras is due to Mariano
Su\'arez-Alvarez \cite{SA11}. The smallest Steiner triple system is
provided by the points and lines in the Fano plane, $\PP^2_{\FF_2}$, and the
3-Calabi-Yau algebra in that case is the one studied by Smith in
\cite{SmG2}.

\begin{proposition}\label{pr.cy-ngk}
  The $3$-Calabi-Yau algebras described in the theorem in the introduction of \cite{e}
satisfy the hypotheses of \Cref{cor.cy} and therefore admit
  $4$-dimensional Artin-Schelter regular central extensions.
\end{proposition}

%%%%%%%%%%%%%%%%%%%%%%%%%%%%%%%%%%%%%%%%%%%%%%%%%%%%%%%%%%%%%%%%%%%%%%%%%%%%%%%%%
\subsection{Generic cases with finite Gelfand-Kirillov dimension}

The generic $3$-dimensional Artin-Schelter  regular algebras of finite Gelfand-Kirillov dimension are listed in \cite[Tables (3.9) and
(3.11)]{AS87}.  In the terminology of \cref{sect.1331}, they have  type $(1221)$ or $(1331)$.

As remarked after \Cref{th.3-stg}, there is a version of \cref{cor.cy} in which the relations for $D(\sfw,\sfp)$ 
are defined by using $f_{k}$ in place of $f_{1}$. If we do that we set $p_{k}:=q_{k}$ instead of $p_{1}:=q_{1}$. 

Let $\zeta_{n}$ be a primitive $n^{th}$ root of unity.

\begin{proposition}\label{pr.as-reg}
For each 3-dimensional regular algebra $A$ in Tables (3.9) and (3.11) of \cite{AS87} and 
each index $k=1,\ldots,n$ of the relation $\pd_k\sfw$ for $A(\sfw)=A$ that we omit when defining $D(\sfw,\sfp)$,
\cref{table-cubic,table-quadratic} list all tuples ${\sfp}=(p_1,\ldots,p_n)$ that satisfy  condition \cref{item.eigenvec} in \cref{th.3-stg}. 
For each such ${\sfp}$, $D(\sfw,{\sfp})$ is a 4-dimensional Artin-Schelter regular normal extension of $A$.
(In \cref{table-cubic,table-quadratic}, $\l$ is an arbitrary non-zero scalar.)
\end{proposition}

\subsection{}

In \cite{LPWZ}, four classes of 4-dimensional Artin-Schelter regular algebras of type
$(12221)$ were discovered. They are labeled by $A(p)$, $B(p)$, $C(p)$,
and $D(v,p)$, with parameters $0\neq p\in \Bbbk$ and $v\in \Bbbk$. 
If the field $\Bbbk$ is algebraically closed, then these algebras are {\it all} the noetherian regular algebras of type $(12221)$ that satisfy the $(m_{2},m_{3})$-\emph{generic} condition  
in \cite{LPWZ}; roughly speaking, these algebras are generic in terms of the $A_{\infty}$-structure on their 
Yoneda Ext-algebras. 
In \cite{LPWZ}, the regularity of the algebras $A(p)$, $B(p)$, $C(p)$ was proved in a computational way using Bergman's diamond lemma
\cite{bgm}, and that of $D(v,p)$ was deduced from the fact that it is an Ore extension of a skew polynomial ring. Among these algebras, all $A(p)$ and $D(v,p)$ appear as $D(\sfw,\sfp)$ in \cref{th.3-stg} so
their regularity may also be proved as a consequence \cref{pr.as-reg}.

\begin{table}
\begin{center}
\begin{tabular}{|c|c|c|c|c|}
\hline
Type & $\vphantom{\Big\vert}(q_1,q_2)$ & $k=1$ & $k=2$
\\
\hline
$\vphantom{\Big\vert}A$ & $(1,1)$ & $(1,\pm 1)$ & $(\pm 1,1)$
\\
$\vphantom{\Big\vert}E$ & $(1,\zeta_3)$ & $(1,\zeta_3^r),\;r=0,1,2$ & $(\zeta_3,\zeta_3)$
\\
$\vphantom{\Big\vert}H$ & $(\zeta_8,-\zeta_8)$ & none & none
\\
$\vphantom{\Big\vert}S_1$ & $(\a,\a^{-1})$ & $(\a,\a^{-1/2})$ & $(\a^{1/2},\a^{-1})$
\\
$\vphantom{\Big\vert}S_2$ & $(\a,-\a^{-1})$ & $(\a,\a^{-1/2})$ & $((-\a)^{1/2},-\a^{-1})$
\\
$\vphantom{\Big\vert}S_2'$ & $(1,-1)$ & $(1,\pm 1)$ & none
\\
\hline
\end{tabular}
\end{center}
\vskip .12in
\caption{Good $(p_1,p_2)$ for the generic cubic AS-regular algebras in \cite[Table (3.9)]{AS87}}
\label{table-cubic}
\end{table}

\begin{proposition}\leavevmode
	\begin{enumerate}
        \item Consider the 3-dimensional cubic regular algebra of type
          $S_{2}$ in Table (3.9) of \cite{AS87}. Then the normal
          extension at $k=2$ with $p_{1}=(-\alpha)^{1/2}$ is the
          algebra $A(-p_{1})$ defined in \cite[Thm.~A]{LPWZ}. The
          normal extension at $k=1$ with $p_{2}=\alpha^{-1/2}$ is
          isomorphic to $A(-p_{2})$ after interchanging the variables $x$
          and $y$.
        \item Consider the 3-dimensional cubic regular algebra of type
          $S_{1}$ in Table (3.9) of \cite{AS87}. Then the normal
          extension at $k=2$ with $p_{1}=\alpha^{1/2}$ is the algebra
          $D(a,-p_{1})$ defined in \cite[Thm.~A]{LPWZ}. The normal
          extension at $k=1$ with $p_{2}=\alpha^{-1/2}$ is isomorphic
          to $D(-ap_{2}^{2},-p_{2})$ after interchanging the variables
          $x$ and $y$.
	\end{enumerate}
\end{proposition}

\begin{proof}
	Let $A$ be the algebra of type $S_{2}$, $k=2$, and $p:=p_{1}=(-\alpha)^{1/2}$. ``The'' relations for $A$ are
	\begin{equation*}
		\begin{cases}
                  f_{1}=xy^{2}+\alpha y^{2}x=xy^{2}-p^{2}y^{2}x,\\
                  f_{2}=\alpha^{2} yx^{2}-\alpha x^{2}y.
		\end{cases}
	\end{equation*}
	Thus,  ``the'' relations for $D(\sfw,\sfp)$ are $f_{1}$ and
	\begin{align*}
          g_{2}&=x(\alpha^{2} yx^{2}-\alpha x^{2}y)-p(\alpha^{2} yx^{2}-\alpha x^{2}y)x\\
               &=p^{2}(x^{3}y+(-p)x^{2}yx+(-p)^{2}xyx^{2}+(-p)^{3}yx^{3}).
	\end{align*}
	This description agrees with the defining relation of $A(-p)$
        up to scalar multiplication. Similar routine calculations settle the other cases.
\end{proof}

\begin{table}
\begin{center}
\begin{tabular}{|c|c|c|c|c|}
\hline
Type & $\vphantom{\Big\vert}(q_1,q_2,q_3)$ & $k=1$ & $k=2$ & $k=3$
\\
\hline
$\vphantom{\Big\vert}A$ & $(1,1,1)$ & $(1,1,1)$, $(1,\zeta_3,\zeta_3^2)$ & $(1,1,1)$, $(\zeta_3^2,1,\zeta_3)$ & $(1,1,1)$, $(\zeta_3,\zeta_3^2,1)$
\\
$\vphantom{\Big\vert}B$ & $(1,1,-1)$ & $(1,1,\pm 1)$ & $(1,1,\pm 1)$ & $(-1,-1,-1)$
\\
$\vphantom{\Big\vert}E$ & $(\zeta_9,\zeta_9^4,\zeta_9^7)$ & none & none & none
\\
$\vphantom{\Big\vert}H$ & $(1,-1,\zeta_4)$ & $(1,1,\pm 1)$, $(1,-1,\pm\zeta_{4})$ & $(-1,-1,\pm 1)$ & $(-\zeta_4,-\zeta_4,\zeta_4)$
\\
$\vphantom{\Big\vert}S_1$ & $(\a,\b,(\a\b)^{-1})$ & $(\a,\l,\l^{-1})$ & $(\l,\b,\l^{-1})$ & $(\l,\l^{-1},(\a\b)^{-1})$
\\
$\vphantom{\Big\vert}S_1'$ & $(\a,\a^{-1},1)$ & $(\a,\a^{-1/3},\a^{1/3})$ & $(\a^{1/3},\a^{-1},\a^{-1/3})$ & $(\l,\l^{-1},1)$
\\
$S_2$ & $(\a,-\a,\a^{-2})$ & $(\a,\pm \a,\a^{-1})$ & $(\pm \a,-\a,-\a^{-1})$ & $\vphantom{\Big\vert}(1,\pm 1,\a^{-2})$, $(-1,\pm 1,\a^{-2})$
\\
\hline
\end{tabular}
\end{center}
\vskip .12in
\caption{Good $(p_1,p_2,p_3)$ for the generic quadratic AS-regular algebras in \cite[Table (3.11)]{AS87}}
\label{table-quadratic}
\end{table}

%%%%%%%%%%%%%%%%%%%%%%%%%%%%%%%%%%%%%%%%%%%%%%%%%%%%%%%%%%%%%%%%%%%%%%%%%%%%%%%%%
%%%%%%%%%%%%%%%%%%%%%%%%%%%%%%%%%%%%%%%%%%%%%%%%%%%%%%%%%%%%%%%%%%%%%%%%%%%%%%%%%

\bibliography{biblio3}
\bibliographystyle{customplain}

\end{document}